\documentclass[11pt]{amsart}

\usepackage {enumerate}
\usepackage[english]{babel}
\usepackage{setspace}
\usepackage{hyperref}
\usepackage{caption}

\newtheorem{theorem}{Theorem}[section]
\newtheorem{lemma}[theorem]{Lemma}
\newtheorem{corollary}[theorem]{Corollary}

\newtheorem{proposition}[theorem]{Proposition}
\newtheorem{example}[theorem]{Example}

\theoremstyle {definition}
\newtheorem{remark}[theorem]{Remark}

\usepackage{tikz}
\usetikzlibrary{decorations.pathreplacing}

\usepackage[margin=1in]{geometry}

\allowdisplaybreaks

\DeclareMathOperator{\area}{area} 
\DeclareMathOperator{\vol}{vol}
\newcommand{\R}{\mathbb{R}}

\DeclareMathOperator{\Rm}{Rm}
\DeclareMathOperator{\supp}{supp}
\DeclareMathOperator{\proj}{proj}
\DeclareMathOperator{\tr}{tr}

\newcommand{\cH}{\mathcal H}

\title [Isoperimetric structure] {Isoperimetry, scalar curvature, and mass in asymptotically flat Riemannian $3$-manifolds} 

\author{Otis Chodosh}
\address{Department of Pure Mathematics and Mathematical Statistics, University of Cambridge, Wilberforce Road, Cambridge CB3 0WB, United Kingdom}
\email{oc249@cam.ac.uk}

\author{Michael Eichmair}
\address{Faculty of Mathematics, University of Vienna, Oskar-Morgenstern-Platz 1, 1090 Vienna, Austria}
\email {michael.eichmair@univie.ac.at}

\author{Yuguang Shi}
\address{Key Laboratory of Pure and Applied mathematics, School of  Mathematical Sciences, Peking University, Beijing, 100871, P. R. China}
\email{ygshi@math.pku.edu.cn}

\author{Haobin Yu}
\address{School of
 Mathematical Sciences, Peking University, Beijing, 100871, P. R. China} 
\email{robinyu@math.pku.edu.cn}

\begin{document}

\begin{abstract} 

Let $(M, g)$ be an asymptotically flat Riemannian  $3$-manifold with non-negative scalar curvature and positive mass. We show that each leaf of the canonical foliation through stable constant mean curvature surfaces of the end of $(M, g)$ is uniquely isoperimetric for the volume it encloses. 
\end{abstract}

\maketitle

\date{}


\section {Introduction}

A complete Riemannian $3$-manifold $(M, g)$ is said to be \emph{asymptotically flat} if there is a compact subset $K \subset M$ and a diffeomorphism 
\begin{align} \label{eqn:chartatinfinity}
M \setminus K \cong \{ x \in \R^3 : |x| > 1/2\}
\end{align}
with 
\begin{align} \label{eqn:AF}
g_{ij} = \delta_{ij} + \sigma_{ij} \qquad \text{ where } \qquad |x|^{|\alpha|}|(\partial^{\alpha} \sigma_{ij}) (x)| = O(|x|^{-\tau}) \quad \text{ as } \quad |x| \to \infty
\end{align}
for some $\tau > 1/2$ and all multi-indices  $\alpha$ with $|\alpha| =0,1, 2,3$. 
We also require that the scalar curvature of $(M, g)$ is integrable. Moreover, the boundary of $M$, if non-empty, is minimal, and there are no closed minimal surfaces in the interior of $M$. Given $\rho > 1$, we use $S_\rho$ to denote the surface in $M$ that corresponds to the centered coordinate sphere $\{x \in \R^3 : |x| = \rho\}$ in the \emph{chart at infinity} \eqref{eqn:chartatinfinity}. We let $B_\rho$ denote the bounded open region in $M$ that is enclosed by $S_\rho$.  

The ADM-mass (after R. Arnowitt, S. Deser, and C. W. Misner \cite{ADM:1961}) of such an asymptotically flat manifold $(M, g)$ is given by 
\[
m_{ADM} =  \lim_{\rho \to \infty } \frac{1}{16 \pi \rho} \int_{\{|x| = \rho\}} \sum_{i, j = 1}^3 \left( \partial_i g_{ij} - \partial_j g_{ii} \right) x^j.
\]
It is independent of the particular choice of chart at infinity \eqref{eqn:chartatinfinity} by work of R. Bartnik \cite{Bartnik:1986}. The fundamental \emph{positive mass theorem}, proven first by R. Schoen and S.-T. Yau \cite{PMT1} using minimal surface techniques and then by E. Witten \cite{Witten:1981} using spinors, asserts that for $(M,g)$ asymptotically flat with non-negative scalar curvature, $m_{ADM} \geq 0$ with equality only when $(M, g)$ is flat Euclidean space. 

Let $V > 0$. Consider
\begin{align} \label{eqn:RV}
\mathcal{R}_V = \{ \Omega : \Omega \subset M \text{ is a compact region with } \partial M \subset \partial \Omega \text{ and} \vol (\Omega) = V\}
\end{align}
and let
\begin{align} \label{eqn:isoperimetricprofile}
A(V) = - \area (\partial M) + \inf \{ \area (\partial \Omega): \Omega \in \mathcal{R}_V\}.
\end{align}
When the scalar curvature of $(M, g)$ is non-negative, then a result of Y. Shi \cite{Shi:2016} combined with an observation in Appendix K of \cite{mineffectivePMT} shows that there is a region $\Omega_V \in \mathcal {R}_V$ that achieves the infimum in \eqref{eqn:isoperimetricprofile}. The proof that such \emph{isoperimetric regions} exist in $(M, g)$ is indirect and offers no real clue as to the position of these regions. The main result of this paper is to show that if $(M, g)$ is not Euclidean space and provided that the volume $V > 0$ is sufficiently large, then $\Omega_V$ is bounded by the horizon $\partial M$ and a stable constant mean curvature surface that belongs to the \emph{canonical foliation} of the end of $M$. In particular, the solution of the isoperimetric problem in $(M, g)$ for large volumes is \emph{unique}.

\begin {theorem} \label{thm:main}
Let $(M, g)$ be a complete Riemannian $3$-manifold that is asymptotically flat at rate $\tau > 1/2$ and which has non-negative scalar curvature and positive mass. There is $V_0 > 0$ with the following property. Let $V \geq V_0$. There is a unique region $\Omega_V \in \mathcal{R}_V$ such that 
\[
\area (\partial \Omega_V) \leq \area (\partial \Omega)
\]
for all $\Omega \in \mathcal{R}_V$. The boundary of $\Omega_V$ consists of $\partial M$ and a leaf of the canonical foliation of the end of $M$.
\end {theorem}

Theorem \ref{thm:main} shows that non-negative scalar curvature, the large scale isoperimetric structure, and --- in view of the results in Appendix \ref{sec:sharpisoperimetric} --- the positive mass theorem are fundamentally related. 

We recall here that  the isoperimetric deficit of a small geodesic ball reflects the sign of the scalar curvature at the center of the ball: less area is needed to enclose a given small amount of volume if the scalar curvature is larger; see \cite {Nardulli:2009} and the references therein. 

The uniqueness of solutions to the isoperimetric problem for large volumes in Theorem \ref{thm:main} is in strong contrast to the non-uniqueness of large stable constant mean curvature surfaces in asymptotically flat manifolds exhibited by the following example constructed by A. Carlotto and R. Schoen in \cite{Carlotto-Schoen}. 

\begin {example} [\cite{Carlotto-Schoen}] There is an asymptotically flat Riemannian metric $g = g_{ij} \, dx^i \otimes dx^j$ on $\R^3$ with non-negative scalar curvature and positive mass and such that $g_{ij} = \delta_{ij}$ on $\R^2 \times (0, \infty)$.
\end {example} 
We emphasize that the examples constructed in \cite{Carlotto-Schoen} are asymptotially flat of rate $\tau < 1$. 

The special case of Theorem \ref{thm:main} where $(M, g)$ is \emph{also} $C^0$-asymptotic to Schwarzschild with positive mass, i.e. where in addition
\begin{align} \label{eqn:asymptoticSchwarzschild}
g_{ij} = \left( 1 + \frac{m}{2 |x|}\right)^4 \delta_{ij} + O (|x|^{-1-\alpha}) \qquad \text{ as } \qquad |x| \to \infty
\end{align}
in the chart \eqref{eqn:chartatinfinity} for some $m > 0$ and  $\alpha > 0$ was proven by M. Eichmair and J. Metzger in \cite{isostructure} using a completely different technique than the approach developed here, building on an ingenious idea of H. Bray \cite{Bray:1997}.\footnote{The expansion \eqref{eqn:asymptoticSchwarzschild} is required to hold up to and including second derivatives in \cite{isostructure}. Owing to the work on the canonical foliaton by C. Nerz \cite{Nerz:2014}, this requirement can be weakened as stated above.} They have extended this result to higher-dimensional asymptotically flat Riemannian manifolds in \cite{hdiso}. These results in \cite{Bray:1997, isostructure, hdiso} make no assumption on the scalar curvature. 

The analogous question is largely open in the asymptotically hyperbolic setting. O. Chodosh \cite{Chodosh:large-iso} has shown that large isoperimetric surfaces are centered coordinate spheres in the special case where the metric is exactly isometric to Schwarzschild-anti-de Sitter outside of a compact set; cf. \cite{Bray:1997}. We also mention the work of J. Corvino, A. Gerek, M. Greenberg and B. Krummel \cite{CorvinoGerekGreenbergKrummel} on exact Schwarzschild-anti-de Sitter preceding \cite{Chodosh:large-iso}, and the subsequent work of D. Ji, Y. Shi, and B. Zhu \cite{Ji-Shi-Zhu:2015} in this direction.  

In the proof of Theorem \ref{thm:main} we use results on the canonical foliation through stable constant mean curvature surfaces of the end of asymptotically flat manifolds with positive mass. We summarize from the literature what is needed here in Appendix \ref{sec:canonicalfoliation}, deferring the reader to \cite{Huang:2011, Nerz:2014, Ma:2016} for the strongest available results. We mention that this rich field departs from the celebrated results of G. Huisken and S.-T. Yau \cite{Huisken-Yau:1996} and J. Qing and G. Tian \cite{Qing-Tian:2007} for data with Schwarzschild asymptotics \eqref{eqn:asymptoticSchwarzschild}. The uniqueness question for large stable constant mean curvature surfaces in asymptotically flat $3$-manifolds with non-negative scalar curvature that intersect the center of $(M, g)$ has been developed in \cite{stablePMT, Carlotto:2014, mineffectivePMT}. The optimal result one can expect here has been given by A. Carlotto, O. Chodosh, and M. Eichmair in \cite{mineffectivePMT}:

\begin{theorem} [\cite{mineffectivePMT}] \label{thm:mineffectivePMT} 
Let $(M, g)$ be a complete Riemannian $3$-manifold that is asymptotically flat with non-negative scalar curvature. Assume that $(M, g)$ contains no properly embedded totally geodesic flat planes along which the ambient scalar curvature vanishes. Let $C \subset M$ be compact. There is $\alpha = \alpha(C) > 0$ so that every connected closed stable constant mean curvature surface $\Sigma \subset M$ with $\area (\Sigma) \geq \alpha$  is disjoint from $C$.
\end{theorem}

We also mention in this context the delicate relationship between far outlying stable constant mean curvature spheres and the role of scalar curvature discovered in the work of S. Brendle and M. Eichmair \cite{offcenter}. The assumption on the non-existence of certain totally geodesic planes in Theorem \ref{thm:mineffectivePMT} is satisfied when $(M, g)$ is asymptotic to Schwarzschild with mass $m > 0$ to second order; cf. the work of A. Carlotto \cite{Carlotto:2014} and Corollary 1.11 in \cite{mineffectivePMT}. We refer the reader to the introduction of \cite{mineffectivePMT} for a recent survey of the literature on the stability based theory. 

We now describe the proof of Theorem \ref{thm:main}. We first recall that the Hawking mass of stable constant mean curvature spheres in (maximal) initial data $(M, g)$ for spacetimes satisfying the dominant energy condition has been proposed by D. Christodoulou and S.-T. Yau \cite{Christodoulou-Yau:1988} as a ``quasi-local" measure of the gravitational field of the spacetime. Now stable constant mean curvature surfaces arise most naturally as boundaries of solutions to the isoperimetric problem. The role of the isoperimetric defect from Euclidean space,
\[
m_{iso} (\Omega)= \frac{2}{\area (\partial \Omega)} \left( \vol (\Omega) - \frac{\area(\partial \Omega)^{3/2}}{6 \sqrt \pi}\right), 
\]
of compact regions $\Omega \subset M$ in the development of quasi-local mass has been proposed and demonstrated by G. Huisken in e.g. \cite{Huisken:iso-mass, Huisken:MM-iso-mass-video}. In particular, the ADM-mass of the initial data (and thus the spacetime evolving from it) is encoded in the isoperimetric profile of $(M, g)$. In fact, 
\[
m_{ADM} = \lim_{V\to\infty} \frac{2}{A(V)} \left(V - \frac{A(V)^{3/2 }}{6\sqrt{\pi}} \right),
\]
as we discuss in Appendix \ref{sec:sharpisoperimetric}. In particular, the isoperimetric defect $m_{iso} (\Omega_V)$ of isoperimetric regions $\Omega_V$ of large volume $V>0$  \emph{must be} close to $m_{ADM}$. Now, as we recall in Section \ref{sec:divergentisos}, large isoperimetric regions $\Omega_V$ in $(M, g)$ look like Euclidean unit balls $B_1(\xi) \subset \R^3$ with center at $\xi \in \R^3$ when scaled by their volume in the chart at infinity \eqref{eqn:chartatinfinity}. When $|\xi| > 1$, we can use a delicate integration by parts inspired by the work of X.-Q. Fan, P. Miao, Y. Shi, and L.-F. Tam in \cite{Fan-Shi-Tam:2009} to relate the isoperimetric defect of $\Omega_V$ to the ``mass integral" of its boundary. Using that the scalar curvature is integrable, one sees that the isoperimetric defect of such a region is close to zero rather than $m_{ADM}$ --- a contradiction. 

When $|\xi| = 1$, the argument becomes much harder. First, we use the recent solution of a conjecture of R. Schoen due to O. Chodosh and M. Eichmair (discussed here at the beginning of Section \ref{sec:divergentisos}) to ensure that either $\Omega_V$ encloses the center of the manifold or that the unique large component $\Omega_V^\infty$ of $\Omega_V$ is far from the center of the manifold, with the distance diverging as $V \to \infty$. In the latter case, assuming also the boundary $\Omega_V^\infty$ is a topological sphere, we combine the Hawking mass estimate due to D. Christodoulou and S.-T. Yau \cite{Christodoulou-Yau:1988} with the monotonicity of the Hawking mass towards $m_{ADM}$ proven by G. Huisken and T. Ilmanen \cite{Huisken-Ilmanen:2001} to obtain strong analytic estimates for $\partial \Omega_V^\infty$. These estimates allow us to compare the isoperimetric deficit of $\Omega_V^\infty$ with that of a large outlying coordinate sphere to conclude as before that it is \emph{too} Euclidean. To handle the case where $\partial \Omega_V^\infty$ has the topology of a torus or where $\Omega_V$ includes the center of $(M, g)$, we combine a careful analysis of the Hawking mass of such surfaces with information about the canonical foliation. The case where $|\xi| < 1$ is covered by the uniqueness of the leaves of the canonical foliation. This is carried out in Sections \ref{sec:outlyingspheres}--\ref{sec:mainproofapproximate}. 

The proof of Theorem \ref{thm:main} sketched above only works when we impose the stronger decay assumptions \eqref{eqn:Masigma}, \eqref{eqn:MaR} on $(M, g)$. (Incidentally, the decay assumptions stated in Theorem \ref{thm:main} are those of the positive mass theorem.) We obtain Theorem \ref{thm:main} in the stated generality from a completely different line of argument that we develop in Section \ref{sec:MCF}. 

In this argument, we study the mean curvature flow of large isoperimetric surfaces. We prove that, upon appropriate rescaling, the flow of such large isoperimetric surfaces converges to the Euclidean flow $\{S_{\sqrt {1 - 4 t}} (\xi)\}_{t \in [0, 1/4)}$ of $S_1(\xi)$ in $\R^3$. When $\xi \neq 0$, part of this flow will be in a shell-like region that avoids the center of the manifold. Using a computation similar to one due to G. Huisken and T. Ilmanen in \cite{Huisken-Ilmanen:2001}, we show that the Hawking mass of the surfaces forming that shell is close to zero. Using this, we apply the monotonicity of the ``isoperimetric defect from Schwarzschild" discovered by G. Huisken and developed for weak mean curvature flow by J. Jauregui and D. Lee in \cite{Jauregui-Lee:2016} in \emph{two} steps to obtain a contradiction. First, we compare with Schwarzschild of mass $m_{ADM}$ until the time when the surfaces have jumped across the center of $(M,g)$. Then, we compare with Schwarzschild of mass $o(1) m_{ADM}$ until the surfaces have all but disappeared. In this argument we only need a very weak characterization of the leaves of the canonical foliation as being unique among stable constant mean curvature spheres in $(M, g)$.

The argument using mean curvature flow is effective, in that it leads to an explicit estimate on the isoperimetric deficit of large outward area-minimizing regions that are close to balls $B_1(\xi)$ when put on the scale of their volume. On the other hand, the analytic argument described above is likely to yield further information about stable constant mean curvature spheres and could also potentially apply to the study of large isoperimetric regions in asymptotically hyperbolic $3$-manifolds, where it is not possible to appeal to scaling. 

\ \\

\noindent {\bf Acknowledgments.} We sincerely thank Hubert Bray, Simon Brendle, Gerhard Huisken, Jan Metzger, Richard Schoen, and Brian White for their example, encouragement, and for their support. We also thank Christopher Nerz for sharing with us his technical expertise about the canonical foliations and  Felix Schulze and Lu Wang for helpful discussions about the  Brakke flow. Otis Chodosh is supported in part by the EPSRC grant EP/K00865X/1. Michael Eichmair is supported in part by the FWF START-Programme. 


\section {Tools}

Estimate \eqref{CYspheres} below is due to D. Christodoulou and S.-T. Yau \cite{Christodoulou-Yau:1988}. A variation of their argument as in Theorem 12 of \cite{Ros:2005} due to A. Ros gives estimate \eqref{CYtori}. 

\begin {lemma} \label{lem:CY}
Let $\Sigma \subset M$ be a connected closed stable constant mean curvature surface in a Riemannian $3$-manifold $(M, g)$. Then 
\begin{align} \label{CYtori}
H^2 \area (\Sigma) + \frac{2}{3} \int_\Sigma (R + |\mathring {h}|^2) d \mu  \leq \frac{64 \pi}{3}.
\end{align}
When $\Sigma$ is a sphere, then
\begin{align} \label{CYspheres}
H^2 \area (\Sigma) + \frac{2}{3} \int_\Sigma (R + |\mathring {h}|^2) d \mu \leq  16 \pi.
\end{align}
Here, $R$ denotes the ambient scalar curvature and $H$ and $\mathring{h}$ denote, respectively, the constant scalar mean curvature and the trace-free part of the second fundamental form of $\Sigma$ with respect to a choice of unit normal, and $d \mu$ is the area element of $\Sigma$ with respect to the induced metric.
\end {lemma}

The elementary fact stated in the lemma below follows from an explicit ``cut and paste" argument by comparison with balls $B_\rho$ for $\rho > 1$ large. 

\begin {lemma} \label{lem:quadraticgrowth}
Let $(M, g)$ be a complete Riemannian $3$-manifold that is asymptotically flat. There is a constant $c > 0$ depending only on $(M, g)$ such that, for every isoperimetric region $\Omega \subset M$,
\[
\area (B_\rho \cap \partial \Omega) \leq c \rho^2
\]
for all $\rho > 1$.
\end {lemma}

The following lemma is a standard consequence of the ``layer-cake representation" of a function. 

\begin {lemma} \label{lem:polydecay}
Let $(M, g)$ be an asymptotically flat $3$-manifold. Let $\Sigma \subset M \setminus K$ be a surface such that, for some $c > 0$, 
\[
\area (B_\rho \cap \Sigma) \leq c \rho^2
\]
for all $\rho > 1$. Then, for $\alpha > 0$ and $1 < \sigma \leq \rho$,
\[
 \int_{\Sigma \cap (B_{\rho} \setminus B_{\sigma})} |x|^{- \alpha} d \mu \leq  \frac {  \area ( \Sigma \cap (B_\rho \setminus B_\sigma)) }{\rho^\alpha}  + c \, \alpha \int_{\sigma}^{\rho} t^{1 - \alpha} \, d t. 
\]
\end {lemma}

For the statement of the next lemma, recall from Section 4 in \cite{Huisken-Ilmanen:2001} that $M$ is diffeomorphic to the complement in $\R^3$ of a finite union of open balls with disjoint closures. Fix a complete Riemannian manifold $(\hat M, \hat g)$ with $\hat M \cong \R^3$ that contains $(M, g)$ isometrically. We think of $M$ as being included in $\hat M$ below. 

\begin {lemma} [\protect{\cite[Section 6]{Huisken-Ilmanen:2001}}] \label{lem:monotonicity} 
Let $(M, g)$ be a complete Riemannian $3$-manifold that is asymptotically flat with non-negative scalar-curvature. Let $\Sigma \subset M$ be a connected closed surface that is outward area-minimizing in $(\hat M, \hat g)$. Then
\begin{align} \label{eqn:monotonicity}
\sqrt {\frac{\area (\Sigma)}{16 \pi} }  \left( 1 - \frac{1}{16 \pi} \int_\Sigma H^2 d \mu \right)  \leq m_{ADM}.
\end{align}
\end {lemma}


\section {Divergent sequences of isoperimetric regions} \label{sec:divergentisos}

The following result due to O. Chodosh and M. Eichmair is included as Corollary 1.13 in \cite{mineffectivePMT}. It is a consequence of the solution of the following conjecture of R. Schoen: The only asymptotically flat Riemannian $3$-manifold with non-negative scalar curvature that admits a non-compact area-minimizing boundary is flat Euclidean space.

\begin {lemma} \label{lem:characterizationlargeiso}
Let $(M, g)$ be a complete Riemannian $3$-manifold that is asymptotically flat with non-negative scalar curvature and positive mass. Let $U \subset M$ be a bounded open subset that contains the boundary of $M$. There is $V_0 > 0$ so that for every isoperimetric region $\Omega_V$ of volume $V \geq V_0$, either $U \subset \Omega_V$ or $U \cap \Omega_V$ is a thin  smooth region that is bounded by the components of $\partial M$ and nearby stable constant mean curvature surfaces.
\end {lemma} 

The conclusion of the lemma clearly fails in Euclidean space. Under the additional assumption that the scalar curvature of $(M, g)$ is everywhere positive, this result was observed by M. Eichmair and J. Metzger as Corollary 6.2 in \cite{isostructure}. Together with elementary observations on the number of components of large isoperimetric regions as in Section 5 of \cite{stablePMT} and the proof of Theorem 1.12 in \cite{mineffectivePMT}, we obtain the following dichotomy for sequences of isoperimetric regions with divergent volumes:

\begin {lemma} \label{lem:decomposition} 
Let $(M, g)$ be a complete Riemannian $3$-manifold that is asymptotically flat with non-negative scalar curvature and positive mass. Let $\Omega_{V_k}$ be an isoperimetric region of volume $V_k$ where $V_k \to \infty$. After passing to a subsequence, exactly one of the following alternatives occurs: 
\begin {enumerate} [(a)]
\item Each $\Omega_{V_k}$ is connected, $(\partial \Omega_{V_k}) \setminus \partial M$ is connected, and the sequence is increasing to $M$.
\item Each $\Omega_{V_k}$ splits into unions of connected components $\Omega_{V_k}^{res}$ and $\Omega_{V_k}^\infty$ where the $\Omega_{V_k}^\infty$ are connected with connected boundary and divergent in $M$ as $k \to \infty$, and where each $\Omega_{V_k}^{res}$ is contained in an $\varepsilon_k$-neighborhood of the boundary of $M$ where $\varepsilon_k \to 0$ as $k \to \infty$.
\end {enumerate}
\end {lemma}

In particular, every isoperimetric region $\Omega_V$ in $(M, g)$ of sufficiently large volume $V > 0$ has \emph{exactly one} large connected component --- either $\Omega_V$ in alternative (a) or $\Omega_V^\infty$ in alternative (b).  

We include several additional observations --- extracted from the proofs of Theorem 1.2 in \cite{hdiso} and Theorem 1.12 in  \cite {mineffectivePMT} ---  about the sequences in Lemma \ref{lem:decomposition}. Let 
\[
\tilde \Omega_{V_k} \subset \{x \in \R^3 : \lambda_k |x| > 1/2\}
\]
be such that
\[
\Omega_{V_k} \setminus K \cong \{ \lambda_k x : x \in \tilde \Omega_{V_k} \} 
\]
where
\[
\lambda_k = \sqrt[3] { (3 V_k)/(4 \pi)}.
\]
Then, possibly after passing to a further subsequence, 
\[
\tilde \Omega_{V_k} \to B_1(\xi)
\]
in $C_{loc}^\infty (\R^3 \setminus \{0\})$ for some $\xi \in \R^3$. In particular, 
\begin{align}
\area (\Sigma_{V_k}) &= 4 \pi \lambda_k^2 (1 + o (1)) \\
H_{\Sigma_{V_k}}  &=  2 (1 + o (1))/\lambda_k
\end{align}
as $k \to \infty$ where $\Sigma_{V_k} = \partial \Omega_{V_k}  \setminus \partial M$. 

We will show in the proof of Theorem \ref{thm:main} that $\xi = 0$. In other words, alternative (b) in Lemma \ref{lem:decomposition} never occurs. \\

In the statement of the following lemma, we use the notation of Lemma \ref{lem:monotonicity}.

\begin {lemma} \label{lem:HIiso}
Let $(M, g)$ be a complete Riemannian $3$-manifold that is asymptotically flat with non-negative scalar curvature.
The outer boundary $\Sigma = \partial \Omega \setminus \partial M$ of the unique large component $\Omega$ of a large isoperimetric region $\Omega_V$ in $(M, g)$ is connected and outward area-minimizing in $(\hat M, \hat g)$. In particular, 
\[
\sqrt {\frac{\area (\Sigma)}{16 \pi} }  \left( 1 - \frac{1}{16 \pi} \int_\Sigma H^2 d \mu \right)  \leq m_{ADM}.
\]
\begin {proof} We have already seen that $\Sigma$ is connected. 
Let $\hat \Omega \subset \hat M$ be the least area enclosure of $\Omega$ in $(\hat M, \hat g)$. Recall from e.g. Theorem 1.3 in \cite{Huisken-Ilmanen:2001} that the boundary $\hat \Sigma$ of $\hat \Omega$ is $C^{1, 1}$ and smooth away from the coincidence set $\hat \Sigma \cap \Sigma$. Assume that $\hat \Omega \neq \Omega$. It follows that the volume of $(M \cap \hat \Omega) \cup \Omega^{res}$ is strictly larger than that of the isoperimetric region $\Omega \cup \Omega^{res}$ so that by the monotonicity of the isoperimetric profile of $(M, g)$ its boundary area is less. A cut-and-paste argument using that the area of $\hat \Sigma$ is less than that of $\Sigma$ shows otherwise --- a contradiction.  
\end{proof}
\end {lemma}


\section{Area and volume of large, outlying coordinate spheres} \label{sec:outlyingspheres}

The computations in this section follow closely the ideas leading to Corollary 2.3 (stated here as Lemma \ref{lem:FSTM}) in \cite{Fan-Shi-Tam:2009}. \\

Let $(M, g)$ be a complete Riemannian $3$-manifold that is asymptotically flat at rate $\tau = 1$. \\

We abbreviate  
\[
n^i (x) = \frac{x^i - a^i}{|x - a|}
\]
throughout. Unless we indicate otherwise, integration is with respect to the Euclidean background metric in the chart at infinity \eqref{eqn:chartatinfinity}. 

\begin{lemma}\label{0m} 
Let $\rho > 0$ and $a \in \R^3$ with $|a| - \rho > 1$. Consider a Euclidean coordinate sphere $S_\rho (a) = \{ x \in \R^3 : |x - a| = \rho\}$. Then
\begin{align} \label{adm}
 \sum_{i, j = 1}^3 \int_{S_{\rho}(a)} \left( \partial_i \sigma_{ij} - \partial_j \sigma_{ii} \right)n_j = O \left( \frac{1}{|a| - \rho}\right).
\end{align}
\end{lemma}
\begin{proof}
The decay assumptions for the metric in the chart at infinity imply that 
\[
R  = \sum_{i,j = 1}^3 (\partial_i \partial_j g_{ij} - \partial_j \partial_j g_{ii}) +O(|x|^{-4}).
\]
Then
\begin{align*}
\int_{B_{\rho} (a)} R = & \sum_{i,j = 1}^3  \int_{B_{\rho} (a)}(\partial_i \partial_j g_{ij} - \partial_j \partial_j g_{ii})+O \int_{B_{\rho} (a)} \frac{1}{|x|^4}\\
= & \sum_{i,j = 1}^3 \int_{S_{\rho}(a)}(\partial_i \sigma_{ij} - \partial_j \sigma_{ii}) n_j + O \int_{B_{\rho} (a)} \frac{1}{|x|^4}
\end{align*}
where integration is with respect to the Euclidean background metric. Now
\[
 \int_{B_{\rho} (a)}\frac{1}{|x|^4} = O \int_{\{x \in \R^3 : |x| \geq |a| - \rho\} } \frac{1}{|x|^4} = O \left( \frac{1}{|a| - \rho}\right).
\]
\end{proof}

\begin{proposition}\label{prop:av01}
We have that 
\begin{align}\label{aat}
\area(S_\rho (a)) &=4\pi \rho^2+\frac{1}{2}\int_{S_{\rho}(a)}
(\delta^{ij} - n^i n^j ) \sigma_{ij}+o(\rho) \\
\label{vat}
\vol(B_\rho(a)) &=\frac{4\pi \rho^3}{3}+\frac{\rho}{4}\int_{S_{\rho}(a)}
(\delta^{ij} - n^i n^j )  \sigma_{ij}+o(\rho^2)
\end{align}
as $\rho \to \infty$ and $|a| - \rho \to \infty$.
\begin{proof} 
Let $t\in[1,\rho]$. Note that 
\begin{align}\label{sal}
\area(S_t(a))=4\pi t^2+\frac{1}{2}\int_{S_{t}(a)}
(\delta^{ij} - n^i n^j ) \sigma_{ij}+ O \int_{S_{t}(a)} \frac{1}{|x|^2}.
\end{align}
Indeed, the area element with respect to the induced metric is given by
 \[
 d\mu= \left(1+ \frac{1}{2} (\delta^{ij} - n^i n^j ) \sigma_{ij}+O(|x|^{-2}) \right) d \overline \mu.
 \]
Now
\begin{align} \label{sas}
\int_{S_{t}(a)}  \frac{1}{|x|^2}
=&\int_{0}^{2\pi}\int_{0}^{\pi}\frac{t^2\sin\phi}{|a|^2+t^2-2|a|t\cos\phi}d\phi d\theta =\frac{\pi t}{|a|}\log \left ( \frac{|a| + t}{|a| - t} \right) =o(t),
\end{align}
giving \eqref{aat}.
Differentiating \eqref{sal}, we obtain that 
\begin{align*}
\partial_t \area (S_t(a))
=&8\pi t+
\frac{1}{2}\int_{S_{t}(a)} n_k \partial_k( (\delta^{ij} - n^i n^j) \sigma_{ij})+
\frac{1}{t}\int_{S_{t}(a)} (\delta^{ij} - n^i n^j)\\
&+ O \int_{S_{t}(a)} |x|^{-3}
+\frac{1}{t} \, O \int_{S_{t}(a)} |x|^{-2}.
\end{align*}
Using that 
\[
n_k \partial_k n^i = 0
\]
for all $i = 1, 2, 3$, we obtain that 
\begin{align*}
\partial_t \area(S_t(a))
=&8\pi t+
\frac{1}{2}\int_{S_{t}(a)} n^k (\delta^{ij} - n^i n^j) \partial_k \sigma_{ij}+
\frac{1}{t}\int_{S_t(a)}(\delta^{ij} - n^i n^j)\sigma_{ij}\\
&+ O \int_{S_{t}(a)} \frac{1}{|x|^3}
+\frac{1}{t} \, O \int_{S_{t}(a)} \frac{1}{|x|^2}.
\end{align*}
Observe that
\begin{align*}
\int_{S_{t}(a)}  n^i n^j n^k \partial_k \sigma_{ij} &= \int_{S_{t}(a)}n^i n^k \partial_k (n^j \sigma_{ij})\\
=&-\int_{S_{t}(a)}(\delta^{ik}-n^i n^k)\partial_k (n^j \sigma_{ij}) + \int_{S_{t}(a)} \delta^{ik} \partial_k ( n^j \sigma_{ij})\\
=&-\frac{2}{t}\int_{S_{t}(a)} n^i n^j \sigma_{ij} +\int_{S_{t}(a)} \delta^{ik} n^j \partial_k \sigma_{ij}  + \frac{1}{t} \int_{S_{t}(a)}(\delta^{ij} - n^i n^j) \sigma_{ij}\\
=&-\frac{2}{t}\int_{S_{t}(a)} n^i n^j \sigma_{ij} +\int_{S_{t}(a)} \delta^{ik} n^j \partial_k \sigma_{ij}  +\frac{1}{t} \int_{S_{t}(a)}(\delta^{ij} - n^i n^j) \sigma_{ij}
\end{align*}
where we have used the first variation formula in the second equality. 
The last two equalities combine to give that
\begin{align}
\partial_t \area(S_t(a))
=&8\pi t+
\frac{1}{2}\int_{S_{t}(a)}  \delta^{ik} n^j(\partial_j \sigma_{ik} - \partial_i \sigma_{kj} ) 
+ \frac{1}{t} \int_{S_{t}(a)}n^i n^j\sigma_{ij}\label{paa}\\
&+\frac{1}{2t}\int_{S_{t}(a)} (\delta^{ij} - n^i n^j)\sigma_{ij}
+ \frac{1}{t} \, O\int_{S_{t}(a)} \frac{1}{|x|^2}
+ O \int_{S_{t}(a)}\frac{1}{|x|^3}\nonumber.
\end{align}
Substituting \eqref{sal} into \eqref{paa} and applying  Lemma \ref{0m} gives that
\begin{align}
\partial_t \area (S_t(a))
=&\frac{\area (S_t(a))}{t}+4\pi t
+ \frac{1}{t} \int_{S_{t}(a)} n^i n^j \sigma_{ij}
+ O \int_{S_{t}(a)}\frac{1}{|x|^3} \label{aa2}\\
&+ \frac{1}{t}\, O \int_{S_{t}(a)} \frac{1}{|x|^2}+o(1).\nonumber
\end{align}
Next we give an estimate of $\vol (B_t(a))$.
By the co-area formula,
\begin{align}
\partial_t \vol (B_t(a))
=&\int_{S_{t}(a)} \frac{t \, d \mu }{ \sqrt {g_{ij} (x^i -a^i) (x^j - a^j) }}\label{va}\\
=&\area (S_t(a))+\frac{1}{2}\int_{S_{t}(a)}n^i n^j \sigma_{ij}
+O \int_{S_{t}(a)} \frac{1}{|x|^2}\nonumber
\end{align}
which in conjunction with \eqref{aa2} yields
\begin{align}
\partial_t \area (S_t(a))
=&\frac{\area (S_t(a))}{t}+4\pi t+
\frac{1}{t}\Big(2 \partial_t \vol (B_t(a))-2\area (S_t(a))\Big) \label{aa3}\\
&+ O \int_{S_{t}(a)} \frac{1}{|x|^3}
+\frac{1}{t} \, O \int_{S_{t}(a)} \frac{1}{|x|^2} + o(1).\nonumber
\end{align}
It follows that
\begin{align*}
\partial_t (t \area (S_t(a)))=4\pi t^2
+2\partial_t\vol (B_t(a))
+t \,  O \int_{S_{t}(a)}\frac{1}{|x|^3} + O \int_{S_{t}(a)}\frac{1}{|x|^2} + o(t).
\end{align*}
Integrating from $1$ to $\rho$ yields
\begin{align} \label{av1}
\rho \area (S_\rho(a))=\frac{4\pi \rho^3}{3}+2\vol (B_\rho(a))+O \int_1^{\rho}\int_{S_{t}(a)} \frac{t}{|x|^3}dt 
+O \int_1^{\rho}\int_{S_{t}(a)} \frac{1}{|x|^2} dt+o(\rho^2).
\end{align}
A direct computation shows
\[
\int_{S_{t}(a)} \frac{1}{|x|^3}
=\frac{2 \pi t}{|a|}\left( \frac{1}{|a| - t}- \frac{1}{|a| + t}\right)
\]
so that
\begin{align}
\int_1^{\rho}\int_{S_{t}(a)} \frac{t}{|x|^3} dt
=o(\rho^2).\label{error1}
\end{align}
Similarly, 
\begin{align}\label{error2}
 \int_1^{\rho}\int_{S_{t}(a)} \frac{1}{|x|^2} dt
=o(\rho^2).
\end{align}
Substituting \eqref{error1} and \eqref{error2} into \eqref{av1} gives \eqref{vat}.
\end{proof} 
\end{proposition}


\section{Isoperimetric deficit of large outlying isoperimetric spheres} \label{sec:isospheres}

Throughout this section, we consider a complete Riemannian $3$-manifold $(M, g)$ that is asymptotically flat at rate $\tau = 1$ and which has non-negative scalar curvature and positive mass $m_{ADM} > 0$. We mention that the results in this section work for $\tau > 3/4$. (This is the threshold for the proof of the key estimate \eqref{bthe}.) As a step in the proof of Theorem \ref{thm:main}, the argument here needs the full strength of the results by S. Ma in \cite{Ma:2016}, which require that $\tau = 1$. This is why we restrict the exposition to this case.  \\

We consider isoperimetric regions $\Omega_V = \Omega^{res} \cup \Omega$ where $\Omega^{res}$ is contained in a small neighbourhood of the horizon and  $\Omega \cap B_{\rho_0} = \emptyset$ for some $\rho_0 > 1$. We assume throughout that the volume $V > 0$ is large. We know from Section \ref{sec:divergentisos} that the boundary $\Sigma = \partial \Omega$ is connected and outward area-minimizing in the sense of Lemma \ref{lem:monotonicity}. We may assume that $\rho_ 0 > 1$ is large. All the error terms in this section are with respect to volume $V \to \infty$. \\

We  assume throughout this section that $\Sigma$ has the topology of a sphere. \\

Let $r > 0$ denote the area radius
\[
\area (\Sigma) = 4 \pi r^2
\]
of $\Sigma$. We use $H > 0$ to denote the mean curvature of $\Sigma$. 

Using \eqref{CYspheres} and Lemma \ref{lem:HIiso}, we find 
\begin{align}
r \int_\Sigma |\mathring h|^2  d \mu \leq 4 8 \pi m_{ADM} \label{eqn:CYHI} 
\end{align}
and 
\begin {align*}
2 \sqrt{1 - 2 m_{ADM}/r } \leq r H  \leq 2. 
\end {align*}
From this, we see 
\begin{align} \label{eqn:mcarea}
|H - 2/r| = O (r^{-2});
\end{align}
cf. \cite[p.\ 425]{Chodosh:large-iso}. 
In conjunction with the results stated in Appendix \ref{sec:gbarg}, with $\tau = 1$, we obtain  
\begin{align}\label{a0l2}
\int_{\Sigma}|\mathring{\overline{h}}|_{\overline{g}}^2d\overline{\mu}
= O \left(\int_{\Sigma}|\mathring{h}|_{g}^2d\mu+
\int_{\Sigma} |h|^2_g |x|^{-2}d\mu+\int_{\Sigma}|x|^{-4}d\mu\right) \\ \nonumber
= O \left(r^{-1} + \rho_0^{-2} \right).
\end{align}

We next recall a consequence of J. Simons' identity for the trace-free part of the second fundamental form.

\begin{lemma} [Cf. Corollary 5.3 in \cite{hdiso}]
There is a constant $c>0$ with the following property. Consider in a Riemannian manifold a two-sided hypersurface with constant mean curvature $H$ and trace-free second fundamental form $\mathring{h}$. Then 
\begin{align} \label{eqn:weakSimons}
2 |\mathring h|^3 + \Delta|\mathring{h}| \geq - c ( H |\mathring{h}|^2 + H |{Rm}| + |\mathring{h}| |Rm| + |\nabla Rm|)
\end{align}
holds weakly, where $\Delta$ is the induced Laplace-Beltrami operator and where $\Rm, \nabla \Rm$ are the ambient Riemannian curvature tensor and its first covariant derivative both restricted along  the surface. 
\end{lemma}
 
\begin{proposition} \label{prop:moser}
There is a constant $c> 0$ depending only on $(M, g)$ such that 
\begin{align}\label{A0}
 |\mathring{h}(x)| \leq c \,  r^{-5/4}
\end{align}
for all $x \in \Sigma$ such that $2 |x| \geq r^{3/4}$.

\begin {proof} 
Assume that the assertion fails with $c = k$  along a sequence of regions $\Omega_k$ with area radius $r_k \to \infty$ and at points $x_k \in \Sigma_k = \partial \Omega_k$ where $r_k^{3/4} \leq 2 |x_k|$. We work in the chart at infinity $\{x  \in \R^3 : |x| > 1/2\}$. If we rescale by $r_k^{-3/4}$ and pass to a subsequence, then the rescaled regions converge in $C^{2, \alpha}_{loc}$ to a half-space in $\R^3 \setminus\{0\}$. Upon further translation by the points $r_k^{-3/4} x_k$ we find surfaces $\tilde \Sigma_k$ in $B_{1/4}(0)$ with $0 \in \tilde \Sigma_k$ that are locally isoperimetric with respect to a metric $\tilde g_k$ on $B_{1/4}(0)$ and such that
\begin{align} \label{eqn:auxcurvature}
r_k |\mathring {\tilde h}_k(0)|^2 \geq k^2 \qquad \text{ and } \qquad r_k \int_{\tilde \Sigma_k} |\mathring {\tilde h}_k|^2 d \tilde \mu_k \leq 48 \pi m_{ADM}.
\end{align}
(The second estimate follows from \eqref{eqn:CYHI}, inclusion, and scaling invariance.) The surfaces $\tilde \Sigma_k $ converges in $C^{2, \alpha}$ to a plane through the origin in $B_{1/4} (0)$. The Riemannian metrics $\tilde g_k$ converge to the Euclidean metric on $B_{1/4}(0)$ with 
\[
\tilde \Rm_k  = O (r_k^{-3/4}) \qquad \text{ and } \qquad \tilde \nabla_k \tilde \Rm_k = O (r_k^{-3/4}).
\]
For large $k$, \eqref{eqn:weakSimons} and \eqref{eqn:auxcurvature} are incompatible with the estimate in Theorem 8.15 in \cite{Gilbarg-Trudinger:1998}.
\end {proof}
\end{proposition}

Using \eqref{eqn:mcarea},  \eqref{eqn:euclidH}, \eqref{eqn:euclidring} and Proposition \ref{prop:moser},  we see that 
\begin{align*}
\overline{H}(x) 
=2/r+O(r^{-3/2}) \qquad \text{ and } \qquad 
|\mathring{\overline{h}} (x)|= O (r^{-5/4})
\end{align*}
for all points $x\in\Sigma$ with $2 |x|\geq r^{3/4}$. In particular, the Euclidean principle curvatures  $\overline{\kappa}_i(x)$ of $\Sigma$ satisfy
\begin{align}\label{princur}
\overline{\kappa}_i(x)= 1/r+O(r^{-5/4})
\end{align}
for $2 |x| \geq r^{3/4}$, where $i = 1, 2$. Using the Gauss-Weingarten relations, we conclude that 
\begin{align}\label{dy}
| \overline \nabla (\overline{\nu}-x/r)| = O (r^{-5/4}) 
\end{align}
on $\Sigma \setminus B_{r^{3/4}/2}$. \\

Let $\Sigma'$ be a connected component of $\Sigma \setminus B_{r^{3/4}}$. 

\begin{lemma}\label{position}
There is $a \in \R^3$ with $|a| > r + r^{3/4}$ such that 
\begin{align}\label{eqn:almostsphere}
\left|\overline \nu (x) - \frac{x-a}{r} \right| = O (r^{-1/4}) 
\end{align}
for all $x \in \Sigma'$. 
\end{lemma}
\begin{proof}

We first show that the diameter of $\Sigma'$ is $O (r)$. We only need to consider the case where $\{x \in \Sigma : r^{3/4} /2 \leq |x| \leq r^{3/4} \}$ is non-empty.
Using the co-area formula and the quadratic area growth of isoperimetric surfaces, we see that
\[
\int_{ r^{3/4}/2}^{ r^{3/4}}
\mathcal{H}^1(\{x \in \Sigma : |x| = \sigma\}) \, d \sigma = O (r^{3/2}).
\]
We can choose a regular value $\sigma$ with $r^{3/4}/2 \leq \sigma \leq r^{3/4}$ such that the curve $\{x \in \Sigma : |x| = \sigma\}$ has length $O (r^{3/4})$. A standard variation of the argument leading to the Bonnet-Myers diameter estimate shows that any two points $p, q\in\Sigma'$ are connected by a curve in $\Sigma'$ whose length is $O(r)$. By integrating \eqref{dy} along such curves, we see that there is $a \in \R^3$ so that \eqref{eqn:almostsphere} holds. Assume now that $|a| \leq  r+ r^{3/4}$. It follows that there is $x_0 \in  \Sigma'$ with $|x_0| = r^{3/4}$. Using \eqref{eqn:almostsphere}, it follows that
\[
|a|\geq|a-x_0|-|x_0|\geq  r-c \, r^{3/4}
\]
where $c > 0$ is independent of $\Sigma$. Replacing $a$ by 
\[
a' = (1 + (C + 2) r^{-1/4}) a
\] 
completes the proof
\end{proof}


Now by \eqref{princur}, there is an open subset $\Gamma \subset \mathbb{S}^2$ and $u \in C^\infty(\Gamma)$ with 
\[
\Sigma'=\{a+u(\theta)\theta : \theta \in \Gamma \}.
\]

Let us assume for definiteness that $a = |a|(0, 0, 1)$. We have the estimate    
\begin{align} \label{eqn:shapegamma}
\mathbb{S}^2 \setminus \Gamma \subset \{ (\sin\phi\cos\theta, \sin\phi\sin\theta,\cos\phi) : \theta \in [0, 2 \pi] \text{ and } \phi \in (\pi - 2 r^{-1/4}, \pi] \}.
\end{align}
We also remark that $\Gamma = \mathbb{S}^2$ when $|a| > 2r$. 

Let $\overline{\nu}$ be the outward pointing unit normal and $\overline{p}$ the induced metric of $\Sigma'$ both computed  with respect to the Euclidean background metric. Then
\[
\overline{\nu} (\theta) = \frac{ u (\theta)  \theta -  (\nabla u ) (\theta)}{ \sqrt { u (\theta)^2 + |(\nabla u )(\theta)|^2  }} \qquad \overline p_{ij} (\theta) = u(\theta)^2 \omega_{ij} + (\partial_i u)(\theta) (\partial_j u ) (\theta) 
\]
where the gradient and its length are both computed with respect to the standard metric $\omega_{ij}$ on  $\mathbb{S}^2 \subset \R^3$ and where $i, j$ are with respect to local coordinates on $\mathbb{S}^2$.
It follows from \eqref{eqn:almostsphere} that
\begin{align}\label{uu}
u=r+O(r^{3/4})
\qquad \text{ and } \qquad \nabla u = O (r^{3/4}).
\end{align}

Note that it also follows that $\Sigma \setminus B_{r^{3/4}}$ is connected (so  $\Sigma' = \Sigma \setminus B_{r^{3/4}}$) since otherwise we would contradict the Euclidean isoperimetric inequality. 


\begin{lemma}
We have 
\begin{align}\label{bthe}
\int_{\{ a + r  \theta \in \R^3 : \theta \in \mathbb{S}^2 \setminus \Gamma\}} \frac{1}{|x|} =o(r).
\end{align}

\begin{proof}
We may assume that $r + r^{3/4} < |a| < 2 r$. Using \eqref{eqn:shapegamma}, we have that
\begin{align*}
\int_{\{ a + r  \theta : \theta \in \mathbb{S}^2 \setminus \Gamma\}} \frac{1}{|x|} = O \left( \int_{0}^{2 \pi} \int_{\pi-2r^{-\frac{1}{4}}}^{\pi} \frac{ r^2 \sin \phi \,  d \phi d \theta}{ \sqrt {|a|^2+r^2-2|a|r\cos\phi } } \right) = o(r)\end{align*}
as claimed.
\end{proof}
\end{lemma}


\begin{proposition}\label{prop:av02}
We have  
\begin{align}
\area (\Sigma) - \overline {\area} (\Sigma) &= \area(S_r(a)) - 4 \pi r^2 +o( r ) \\    \vol (\Omega) - \overline \vol (\Omega)  & =  \vol(B_r(a))  - \frac{4 \pi r^3}{3} + o (r^2).
\end{align}

\begin{proof}
Set
$\Sigma''=\Sigma-\Sigma'=\Sigma\cap B_{r^{3/4}}$.
Then
\begin{align} \label{adef}
\area (\Sigma) = & \,  \overline {\area} (\Sigma)+\frac{1}{2}\int_{\Sigma} (\delta^{ij} - \overline \nu^i \overline \nu^j)\sigma_{ij} + O \int_{\Sigma} \frac{1}{|x|^2} \\
=& \,  \overline {\area} (\Sigma)+\frac{1}{2}\int_{\Sigma'}(\delta^{ij} - \overline \nu^i \overline \nu^j) \sigma_{ij}  + O \int_{\Sigma''} \frac{1}{|x|}  + O \int_{\Sigma} \frac{1}{|x|^2} .\nonumber
\end{align}
It follows from Lemmas \ref{lem:quadraticgrowth} and \ref{lem:polydecay} that
\begin{align*}
\int_{\Sigma'} \frac{1}{|x|} = O(r) \qquad \int_{\Sigma''}\frac{1}{|x|} = O(r^{3/4}) \qquad  \int_{\Sigma}\frac{1}{|x|^2}=O( \log r)
\end{align*}
so that 
\begin{align}\label{asi}
\area (\Sigma) = \, \overline{\area} (\Sigma)  +\frac{1}{2}\int_{\Sigma'}(\delta^{ij} - \overline \nu^i \overline \nu^j)\sigma_{ij} +o(r).
\end{align}
Now 
\begin{align} \label{al3}
\int_{x \in \Sigma'}((\delta^{ij} - \overline \nu^i \overline \nu^j) \sigma_{ij} )(x) \\
=\int_{\theta \in \Gamma} \left( (\delta^{ij} - \overline \nu^i \overline \nu^j)  \sigma_{ij} \right) (a + u (\theta) \theta) \overbrace {u(\theta)  \sqrt {  u (\theta)^2 + |(\nabla u)(\theta)|^2  }}^{\text{area element}} \nonumber  \\
=r^2 \int_{\theta \in \Gamma} \left( (\delta^{ij} - \overline \nu^i \overline \nu^j)  \sigma_{ij} \right) (a + u (\theta) \theta) (1+O(r^{-1/4})) \nonumber\\
=O(r^{3/4}) + r^2 \int_{\theta \in \Gamma} \left( (\delta^{ij} - \overline \nu^i \overline \nu^j)  \sigma_{ij} \right) (a + u (\theta) \theta) \nonumber
\end{align}
where  we have used \eqref{uu} in the second equality. 
On the other hand,
\[
\int_{\theta \in \Gamma} \left( (\delta^{ij} - \overline \nu^i \overline \nu^j)  \sigma_{ij} \right) (a + u (\theta) \theta) = o (r) +  \int_{\theta \in \Gamma} (\delta^{ij} -  \theta^i \theta ^j)  \sigma_{ij} (a + r \theta)
\]
since 
\begin{align}\label{al4} 
&\int_{\theta \in \Gamma} (\delta^{ij} - \overline \nu^i \overline \nu^j )  \sigma_{ij}(a + u (\theta) \theta)  -\int_{\theta \in \Gamma} (\delta^{ij} - \theta^i \theta^j) \sigma_{ij}(a + r \theta) \\
= & r^{-1/4} \, O \,  \int_{\theta \in \Gamma} \frac{1}{| a+ r \theta|}+ r^{3/4} \, O \, \int_{\theta \in \Gamma}  \frac{1}{|a + r \theta|^2} = O ( r^{3/4} \log r ) =o(r).\nonumber
\end{align}
Substituting \eqref{al3} and \eqref{al4} into \eqref{asi} gives
\begin{align}\label{as1}
\area (\Sigma) = \overline{\area} (\Sigma) + \frac{r^2}{2}\int_{\theta \in \Gamma}(\delta^{ij} - \theta^i \theta^j) \sigma_{ij} (a + r \theta)+o(r).
\end{align}
A direct computation shows that
\begin{align} \label{asa}
\area (S_r(a))
=&4\pi r ^2+ \frac{1}{2} \int_{S_r(a)} (\delta^{ij} - n^i n^j) \sigma_{ij}  + O\int_{S_{ r }(a)}  \frac{1}{|x|^2} \nonumber\\
=&4\pi r ^2+ \frac{r^2}{2} \int_{\Gamma}(\delta^{ij} - \theta^i \theta^j) \sigma_{ij} (a + r \theta) + O \int_{S_{ r }(a)-\Gamma} \frac{1}{|x|}+o( r ).\nonumber\\
=&4\pi r ^2+ \frac{r^2}{2}  \int_{\Gamma}(\delta^{ij} - \theta^i \theta^j) \sigma_{ij} (a + r \theta)  +o( r )
\end{align}
where we have used \eqref{bthe} in the last equality.
Combining \eqref{as1} and \eqref{asa} yields
\[
\area (\Sigma) - \overline {\area} (\Sigma)= \area(S_r(a)) - 4 \pi r^2 +o( r ).
\]

To give an estimate of $\vol (\Omega)$, we assume for definiteness that $a= |a|(0,0,1)$ where $|a| > r$. Set 
\[
\Omega'=\{x \in \Omega: |x|<r^{3/4}\} \qquad \text{ and } \qquad \Omega''=\{x \in \Omega : |x|\geq r^{3/4}\}.
\]
Then
\begin{align*}
&\vol (\Omega)-\vol (B_r(a)) \\
=&\int_{\Omega}\sqrt{\det(g_{ij})} -\int_{B_{r}(a)}\sqrt{\det(g_{ij})} \\
= &  \overline \vol (\Omega) - \overline \vol (B_r(a)) +  O \int_{\{x : r - c r^{3/4} \leq |x - a| \leq r + c r^{3/4}\}   }\frac{1}{|x|}  + O \int_{\{x \in \Omega : |x| < r^{3/4}\} } \frac{1}{|x|}\\
=&  \overline \vol (\Omega) - \overline \vol (B_r(a)) + o (r^2)
\end{align*}
where we have used \eqref{uu} in the third inequality.
\end{proof}
\end{proposition}


We arrive at the main results of this section, asserting that the isoperimetric deficit of large outlying isoperimetric spheres is \emph{very close} to Euclidean. The strategy of the approximation argument we have used here is illustrated in Figure \ref{fig:approx}.

\begin {corollary} We have 
\begin {align}
\frac{2}{\area(\Sigma)} \left( \vol (\Omega) - \frac{\area(\Sigma)^{3/2}}{6 \sqrt {\pi}} \right)  \leq o (1).
\end {align}
\begin {proof} 
We abbreviate 
\[
z = \int_{S_r(a)} (\delta^{ij} - \overline \nu^i \overline \nu^j) \sigma_{ij}.
\]
Note that $z = O (r)$ and that $\overline{\area} (\Sigma) = 4 \pi r^2 + o (r^2)$. Using Propositions \ref{prop:av01} and \ref{prop:av02} and also the Euclidean isoperimetric inequality in the last step, we obtain that 
\begin {align*}
\vol (\Omega) - \frac{\area(\Sigma)^{3/2}}{6 \sqrt {\pi}}  \\ = \overline{\vol} (\Omega) + \frac{z \,  r}{4} - \frac{\overline \area (\Sigma)^{3/2}} {6 \sqrt \pi} \left ( 1 + \frac{z \,  r}{2 \overline {\area} (\Sigma)} + o \left( \frac{r}  { \overline {\area}( \Sigma) } \right)  \right)^{3/2} +  o (r^2) \\ 
=   \overline{\vol} (\Omega) - \frac{\overline \area (\Sigma)^{3/2}} {6 \sqrt \pi} + o (r^2)  \leq o (r^2). 
\end {align*}
\end {proof}
\end {corollary}


The above arguments yield the following proposition.

\begin {proposition} \label{prop:nooffcentersphere} Let $(M, g)$ be a complete Riemannian $3$-manifold that is asymptotically flat at rate $\tau = 1$ and which has non-negative scalar curvature and positive mass. There does not exist a sequence $\{\Omega_{V_k}\}_{k=1}^\infty$ as in alternative (b) of Lemma \ref{lem:decomposition} such that the boundary of the unique large component of each $\Omega_{V_k}$ has the topology of a sphere.
\end {proposition}


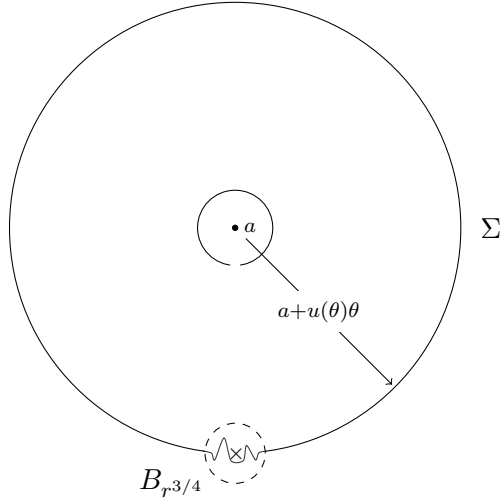
\begin{figure} 
\begin{tikzpicture}
	\node at (0:3.4) {$\Sigma$};
	
	\draw (0,0) circle (3);
	\filldraw [white] (270:3) circle (.4);
	\draw [smooth, opacity=.8] plot coordinates {(261.8:3) (264:3.015) (265:3.1) (267:2.8) (269:3.1) (272:3.1) (273:2.9) (275:3.1) (276:3.01) (278:3)};

	\node at (256:3.5) {$B_{r^{3/4}}$};
	\draw [dashed] (270:3) circle (.4);
	\node at (270.3:3) {$\scriptstyle{\times}$};
	
	\filldraw (0,0) circle (1pt);
	\node at (.2,0) {$\scriptstyle{a}$};
	
	\draw (0,0) circle (.5);
	
	\filldraw [white] (270:.5) circle (.06); 
	
	\draw [->] (315:.2) -- (315:2.95) node [midway,fill=white] {$\scriptstyle{a+u(\theta) \theta}$};

\end{tikzpicture}
\caption{When the boundary $\Sigma$ of the large component of an isoperimetric region has genus zero, we show that $\Sigma$ is \emph{very} close to a sphere $S_{r}(a)$ outside of $B_{r^{3/4}}$; note that ``$\times$'' represents the origin here. This allows us to approximate the isoperimetric deficit of $\Sigma$ by that of $S_{r}(a)$. We show that in the scenario depicted here, the isoperimetric deficit of $S_{r}(a)$ (and thus of $\Sigma$) is too close to Euclidean.}
\label{fig:approx}
\end{figure}


\section {Hawking mass of outlying stable constant mean curvature tori} \label{sec:mHtori}

The barred quantities in the statement of the next lemma are with respect to the Euclidean background metric in the chart at infinity \eqref{eqn:chartatinfinity}, as  in Appendix \ref{sec:gbarg}. 

\begin {lemma} \label{lem:estimateHawkingtori}
Let $(M, g)$ be a complete Riemannian $3$-manifold that is asymptotically flat at rate $\tau >1/2$. There are $\rho_0 > 1$ and $c > 0$ depending only on $(M, g)$ with the following property. Let $\Sigma \subset M $ be a closed surface such that $B_\rho \cap \Sigma = \emptyset$ for some $\rho \geq \rho_0$. Then 
\[
\left | \int_\Sigma H^2 d \mu - \int_\Sigma \overline H^2 d \overline \mu \right| \leq  c \rho^{-\tau} \int_\Sigma |h|^2 d \mu.  
\]
\end {lemma}

Recall that there is $\delta > 0$ such that 
\[
\int_{\overline \Sigma} \overline H^2 d \overline \mu \geq 16 \pi +  \delta
\]
for every closed surface $\overline \Sigma \subset \R^3$ of \emph{positive} genus.\footnote{In fact, one has that $
\int_{\overline \Sigma} \overline H^2 d \overline \mu \geq 8 \pi^2$ by the solution of the Willmore conjecture by F. Marques and A. Neves \cite{Marques-Neves:Willmore}.
} In conjunction with Lemma \ref{lem:estimateHawkingtori} and \eqref{CYtori}, we obtain the following result.

\begin {proposition} \label{prop:Hawkinggenus}
Let $(M, g)$ be a complete Riemannian $3$-manifold that is asymptotically flat and which has non-negative scalar curvature.  There is $\rho_0 > 1$ with the following property. Let $\Sigma \subset M$ be a connected closed stable constant mean curvature surface of positive genus and mean curvature $H >0$. Then 
\[
H^2 \area (\Sigma) \geq 16 \pi.
\] 
provided that $\Sigma \cap B_{\rho_0} = \emptyset$. 
\end {proposition}


\section {Proof of Theorem \ref{thm:main} assuming stronger decay \eqref{eqn:Masigma} and \eqref{eqn:MaR}} \label{sec:mainproofapproximate}

Our strategy here is similar to that of Section 9 in \cite {Chodosh:large-iso}. By contrast, we work on the level of the first derivative of the isoperimetric profile, rather than on the level of second derivatives. 

\begin {proposition} 
Let $(M, g)$ be a complete Riemannian $3$-manifold that is asymptotically flat at rate $\tau = 1$ and which has non-negative scalar curvature and positive mass.  There is a sequence of isoperimetric regions $\{ \Omega_{V_k} \}_{k=1}^\infty$ as in alternative (a) of Lemma \ref{lem:decomposition} such that the boundary of each $\Omega_{V_k}$ has the topology of a sphere. 
\begin {proof} Suppose not. We know from Proposition \ref{prop:nooffcentersphere} that for some $V_0> 0$ the (unique) large component of every isoperimetric region $\Omega_V$ of volume $V \geq V_0$ has boundary $\Sigma_V$ of positive genus. Moreover, the surfaces $\Sigma_V$ diverge in $(M, g)$ as $V \to \infty$. Let  $H_V > 0$ denote the mean curvature of $\Sigma_V$. Let $V > V_0$ be such that the isoperimetric profile is differentiable at $V$. Then 
\[
A' (V) = H_V \qquad \text{ and } \qquad 16 \pi - A'(V)^2 A(V) \leq 16 \pi - H_V^2 \area(\Sigma_V) \leq 0.
\]
provided that $V_0 > 0$ is sufficiently large. The derivative of the absolutely continuous function 
\[
W \mapsto W - \frac{A(W)^{3/2}}{6 \sqrt \pi}
\]
at such volumes $V > V_0$ equals
\[
1 - \frac{A'(V) \sqrt  {A(V)}}{4 \sqrt \pi} = \frac{1}{16 \pi} \frac{16 \pi - A'(V)^2 A(V)}{1 + A'(V) \sqrt {A(V)} / \sqrt {16 \pi} } \leq 0.
\]
In particular, 
\[
\limsup_{V \to \infty} \frac{2}{A(V)} \left ( V - \frac{A(V)^{3/2}}{ 6\sqrt { \pi}} \right) \leq 0
\]
contradicting Theorem \ref{prop:huisk-iso-mass}. 
\end{proof}
\end {proposition}

\begin {proof} [Proof of Theorem \ref{thm:main} with stronger decay  \eqref{eqn:Masigma} and \eqref{eqn:MaR}]

We consider the canonical foliation 
\[
\{\Sigma^H : H \in (0, H_0)\}
\]
of $M \setminus C$ discussed in Appendix \label{sec:canonicalfoliation}. Let $\Omega^H$ be the compact  region bounded by $\partial M$ and $\Sigma^H$.  Using the above and the results from Appendix  \ref{sec:canonicalfoliation}, we see that the set 
\[
J = \{ \vol (\Omega^H)  : H \in (0, H_0) \text{ and } \Omega^H \text{ is an isoperimetric region}\}
\]
is closed in $(0, H_0)$ and unbounded. Assume that there are non-empty intervals $(L_k, R_k) \subset \R \setminus J$ such that $L_k, R_k \in J$ and $L_k \to \infty$. Choose a volume $V_k \in (L_k, R_k)$ such that $A'(V_k)$ exists. (Such volumes $V_k$ lie dense in the interval.) Let $\Omega_k$ be the (unique) large component of an isoperimetric region of volume $V_k$ and let $\Sigma_k = \partial \Omega_k \setminus \partial M$. Note that $\Sigma_k$ has positive genus. Let $H_k > 0$ be its mean curvature. We have that
\[
A'(V_k) = H_k \qquad \text{ and } \qquad 16 \pi - A'(V_k)^2 A(V_k) \leq 16 \pi - H_k^2 \text{area} (\Sigma_k) \leq 0.
\]
Observe that
\[
\sqrt {\frac{A(L_k)}{16 \pi}}\left( 1-  \frac{A'{}^+(L_k)^2 A(L_k)}{16 \pi} \right) \leq \lim_{V \searrow L_k} \sqrt {\frac{A(V)}{16 \pi}} \left( 1 - \frac{A'{}^+(V)^2 A(V)}{16 \pi} \right) \leq0
\]
because the isoperimetric profile is continuous and
\[
\lim_{V \searrow L_k} A'{}^+(V) \leq A'{}^+ (L_k).
\]
Let $H(L_k) \in (0, H_0)$ be such that $\vol (\Omega^{H(L_k)}) = L_k$.  Note that $\Sigma^{H(L_k)}$ has least area for the volume it encloses since $L_k \in J$. In particular,
\[
\area(\Sigma^{H(L_k)}) = A(L_k) \qquad \text{ and } \qquad A' {}^+ (L_k) \leq H(L_k) \leq A'{}^- (L_k)
\]
so that 
\[
\sqrt {\frac{\area (\Sigma^{H(L_k)})}{16 \pi}}\left( 1 - \frac {H(L_k)^2 \area (\Sigma^{H(L_k)})}{16 \pi} \right) \leq\sqrt {\frac{A(L_k)}{16 \pi}} \left( 1 - \frac{A'{}^+(L_k)^2 A(L_k)}{16 \pi} \right).
\]
We then have, using \eqref{eqn:ADMHawking}, 
\begin{align*}
0 &< m_{ADM} = \lim_{k \to \infty} \sqrt {\frac{\area (\Sigma^{H(L_k)})}{16 \pi}} \left( 1 - \frac {H(L_k)^2 \,  \area (\Sigma^{H(L_k)})}{16 \pi} \right) \\ &\leq \liminf_{k \to \infty}\sqrt {\frac{A(L_k)}{16 \pi}}  \left( 1 - \frac{A'{}^+(L_k)^2 A(L_k)}{16 \pi} \right)  \leq 0.
\end{align*}
This contradiction shows that $J$ is connected at infinity. In other words, all sufficiently far out leaves of the canonical foliation are isoperimetric for the volume they enclose. In particular, the isoperimetric profile is smooth for all sufficiently large volumes. \\

The argument above does not by itself show that a far-out leaf of the canonical foliation is \emph{uniquely} isoperimetric for the volume it encloses. However, it follows from the regularity of the profile that any two  isoperimetric regions with the same large volume have the same boundary area \emph{and} mean curvature. In particular, they have the same \emph{positive} Hawking mass. (The Hawking mass approaches $m_{ADM}$ as the volume becomes larger.) It also follows that if $\Omega_V$ is an isoperimetric region in $(M, g)$ of volume $V \geq V_0$ where $V_0 > 0$ is sufficiently large  whose boundary is not a leaf of the canonical foliation, then the boundary $\Sigma$ of the (unique) large component of $\Omega_V$ has \emph{positive} genus. It follows from Lemma \ref{lem:characterizationlargeiso} that $\Sigma$ is outlying. Proposition \ref{prop:Hawkinggenus} then implies that 
\[
\sqrt \frac{\area(\partial \Omega_V)}{16 \pi} \left( 16 \pi - H(V)^2 \area (\partial \Omega_V) \right) \leq \sqrt \frac{\area(\partial \Omega _V)}{16 \pi} \left( 16 \pi - H(V)^2 \area (\Sigma)\right) \leq 0, 
\]
which is a contradiction. 
\end {proof}


\section{Mean curvature flow of large isoperimetric regions} \label{sec:MCF}

Let $(M, g)$ be a complete Riemannian $3$-manifold that is asymptotically flat at rate $\tau > 1/2$ and which has non-negative scalar curvature and positive mass $m_{ADM} > 0$.

Let $\Omega_{V_k} $ be isoperimetric regions of volumes $V_k \to \infty$. Let $\Omega_k$ be the unique large component of $\Omega_{V_k}$. We recall from Section \ref{sec:divergentisos} that $\Omega_k$ is connected with connected outer boundary $\partial \Omega_k \setminus \partial M$, and that $\Omega_k$ is outer area-minimizing in $(M, g)$.

Let $\{ \Omega_{k}(t) \}_{t \geq 0}$ denote the level set flow with initial condition $\Omega_{k}$. Then $\Omega_{k}(t)$ is mean-convex in the sense of \cite [p. 670]{White:size-singular} by Theorem 3.1 in \cite{White:size-singular}. Moreover, by Theorem 5.1 in  \cite{White:size-singular}, the $n$-rectifiable Radon measures 
\begin{align} \label{eqn:mu}
 \mu_{k}(t) = \mathcal{H}^{2}\lfloor \partial^{*}\Omega_{k}(t)
\end{align}
define an integral Brakke flow $\{\mu_k (t)\}_{t \geq 0}$ in $(M, g)$. 

\begin{lemma}\label{lem:quad-area-bds-MCF-iso}
There is a constant $c > 0$ depending only on $(M, g)$ so that 
\[
\area (B_\rho \cap \partial^{*}\Omega_{k}(t)) \leq c  \rho^{2}
\]
for all $\rho\geq 1$ and $j \geq 1$ and $t \geq 0$. 
\end{lemma}
\begin{proof}
By \cite[Theorem 3.5]{White:size-singular}, $\Omega_{k}(t)$ is outward area-minimizing in $\Omega_{k}$. Combined with the fact that $\Omega_{k}$ is outward area-minimizing, we see that $\Omega_{k}(t)$ is outward area-minimizing in $(M,g)$. The claim follows from comparison with coordinate spheres. 
\end{proof}

We may view $\mu_k (t) \lfloor (M \setminus K)$ as a measure on $\{x \in \R^3 : |x| > 1/2\}$ using the chart at infinity \eqref{eqn:chartatinfinity}. In fact, consider the map 
\[
\eta_k : \R^3 \to \R^3 \qquad \text{ given by } \qquad x \mapsto x / \rho_k
\]
and the rescaled measures 
\begin{align} \label{eqn:tildemu}
\tilde \mu_k (t) = \eta_{k*} \left( \mu_k (\rho_k^2 \,  t) \lfloor (M \setminus K) \right)
\end{align}
on $\{x \in \R^3 : \rho_k |x| > 1/2\}$ where $\rho_{k} > 0$ is such that 
\[
\vol (\Omega_{k}) = \frac{4 \pi}{3}  \rho_{k}^{3}.
\]
Then $\{\tilde \mu_k (t)\}_{t \geq 0}$ is a Brakke motion on $\{x \in \R^3 : \rho_k |x| > 1/2\}$ with respect to the metric 
\[
\tilde g_k (x) = \sum_{i, j = 1}^3 g_{ij} (\rho_k x) dx^i \otimes dx^j
\]
on $\{x \in \R^3 : \rho_k |x| > 1/2\}$. As $k \to \infty$, $\tilde g_k$ converges to the standard Euclidean inner product in $C^\infty_{loc} (\R^3 \setminus \{0\})$. We let $\tilde \Omega_k$ be the subset of $\{x \in \R^3 : \rho_k \, |x| > 1/2\}$ such that 
\[
\Omega_k \setminus K \cong \{\rho_k  x : x \in \tilde \Omega_k\}.
\]
We also let $\tilde \Omega_k(t)$ be the subset of $\{x \in \R^3 : \rho_k |x| > 1/2\}$ such that 
\begin{align} \label{eqn:tildeomega}
\Omega_k(t) \setminus K \cong \{ \rho_k  x : x \in \tilde \Omega_k(t) \}.
\end {align} 
By the remarks following Lemma \ref{lem:decomposition}, there is $\xi \in \R^3$ such that, upon passing to a subsequence, 
\[
\tilde \Omega_k \to B_1(\xi) \qquad \text{ in }  \qquad C^\infty_{loc} (\R^3 \setminus \{0\})
\]
as $k \to \infty$. Our goal will be to show that $\xi = 0$.

\begin{proposition}
There is an integral Brakke flow $\{\mu (t)\}_{t\geq0}$ on $\R^{3}\setminus\{0\}$ with the following three properties. 
\begin{enumerate} [(1)]
\item There is a subsequence $\ell(k)$ of $k$ such that, for all $t \geq 0$,
\[
\tilde\mu_{\ell(k)}(t) \rightharpoonup \mu (t)
\]
as Radon measures on $\R^3 \setminus \{0\}$.
\item For almost every $t \geq 0$, there is a subsequence $\ell(k, t)$ of $\ell(k)$ such that 
\[
V_{\tilde \mu_{\ell(k, t)}(t)} \rightharpoonup V_{\mu(t)}
\]
as varifolds above $\R^{3}\setminus\{0\}$. 
\item There is a constant $c>0$ so that  
\[
\mu(t)(B_{\rho}(0)) \leq c \rho^2
\]
for all $\rho > 0$ and $t \geq 0$. 

\end {enumerate} 
\end{proposition}
\begin{proof}
The first two claims follow from T. Ilmanen's compactness theorem for integral Brakke flows, Theorem 7.1 in  \cite{Ilmanen:ellipticreg}. This result is only stated for sequences of Brakke flows with respect to a fixed complete Riemannian metric in \cite{Ilmanen:ellipticreg}. However, the same proof as in \cite{Ilmanen:ellipticreg} applies in the present setting. The quadratic area bounds carry over from Lemma \ref{lem:quad-area-bds-MCF-iso}.
\end{proof}

In view of Proposition \ref{prop:extend-brakke-pt}, it is clear now that $\{\mu(t)\}_{t\geq0}$ extends to an integral Brakke flow in $\R^{3}$ with initial condition 
\[
\mu(0) = \mathcal{H}^{2}\lfloor S_{1}(\xi).
\]
Proposition \ref{prop:standardlimitflow} shows that such a Brakke motion follows classical mean curvature flow -- except possibly for sudden extinction:
\[
\mu(t) = \cH^2\lfloor S_{\sqrt {1 - 4 t}} (\xi)
\] 
for all $t \in [0, T]$ where $T \in [0, 1/4)$. The particular flow at hand is constructed as the limit of level set flows. We use spherical barriers to show that the limiting flow cannot disappear suddenly, i.e. that $T = 1/4$.

\begin{lemma} 
We have that $\mu(t) = \mathcal{H}^{2}\lfloor S_{\sqrt{1-4t}}(\xi)$ for all $t \in [0,1/4)$. 

\begin{proof} 
If not, then there is $T \in [0, 1/4)$ so that $\mu(t) = \mathcal{H}^{2}\lfloor S_{\sqrt{1-4t}}(\xi)$ for $t \in [0,T]$ and $\mu(t) = 0$ for $t> T$. We will prove the result for $|a| \geq 1$ and leave the straightfoward modification to $|a| < 1$ to the reader. 

Assume that $T = 0$. Let $\varepsilon >0$ be small. Upper semi-continuity of density (cf.\  \cite[Corollary 17.8]{Simon:GMT}) implies that $B_{\sqrt {1 - 4 \varepsilon}} (\xi) \subset \tilde \Omega_k(0) = \tilde \Omega_k$ for all sufficiently large $k$. Using that $\tilde g_k$ converges to the standard Euclidean inner product in $C^2_{loc} (\R^3 \setminus \{0\})$ and the avoidance principle for the level set flow, we see that $B_{\sqrt {1 - 9 \varepsilon}} (\xi) \subset \tilde \Omega_k(\varepsilon)$ provided that $k$ is sufficiently large. Recall that $\tilde \mu_k(t) = \cH^2 \lfloor \partial^* \tilde \Omega_k(t)$. We obtain a contradiction with the assumption that $\tilde \mu_k(\varepsilon) \rightharpoonup \tilde \mu(\varepsilon) = 0$.

Assume now that $T \in (0, 1/4)$. Let $0 < \varepsilon < (1 - 4 T)/100$. Upper semi-continuity of Gaussian density (cf.\ \cite{Ilmanen:singularities}) implies that $B_{\sqrt {1 - 4 T - 4 \varepsilon}} (\xi) \subset \tilde \Omega_k(T)$ for all $k$ sufficiently large. Arguing as in the previous case, we see that $B_{\sqrt {1 - 4 T - 9 \varepsilon}} (\xi) \subset \tilde \Omega_k(T+\varepsilon)$. This is a contradiction for the same reason as before. 
\end{proof}
\end{lemma}


Using B. White's version \cite{White:regularityMCF} of K. Brakke's regularity theorem \cite{Brakke} for mean curvature flow, we obtain the following 

\begin{corollary} \label{coro:limit-rescaled-MCF}
Let $(x, t) \in (\R^3 \setminus \{0\}) \times [0, \infty)$ with $(x, t) \neq (\xi, 1/4)$. There is a neighorhood of $(x, t)$ in $\R^3 \times \R$ where $ \{ \tilde \Omega_k (t) \}_{t \geq 0}$ defines a classical mean curvature flow with respect to the Riemannian metric $\tilde g_k$ provided that $k$ is sufficiently large. These flows converge to the shrinking sphere $S_{\sqrt{1-4t}}(\xi)$ as $k \to \infty$ locally smoothly away from the spacetime set $\{0\}\times [0,\infty) \cup \{(\xi,\frac{1}{4})\}$.
\end{corollary}


We define the \emph{disconnecting time} for the rescaled flow by 
\[
\tilde T(|\xi|) =
\begin{cases}
 \frac{1-|\xi|^{2}} {4} & |\xi| < 1\\
0 & |\xi| \geq 1.
\end{cases}
\]
Note that the bulk of $\Omega_{k}(t)$ is disjoint from the center of $(M, g)$ after time $t = \rho_k^2 \,  \tilde T(|\xi|) (1 + o(1))$. 

Assume now that $\xi \neq 0$. Choose $\varepsilon > 0$ such that 
\begin{align} \label{eqn:choiceepsilon}
100 \varepsilon <    1 -  4\tilde T(|\xi|).
 \end{align}
We can make this choice such that  
\[
t_{k} = \rho_{k}^{2}(\varepsilon + \tilde T(|\xi|))
\]
and
\[
T_{k} = \rho_{k}^{2}(1/4-\varepsilon).
\]  
are smooth times for all the level set flows $\{\Omega_k(t)\}_{t \geq 0}$. Indeed, by the work of B. White  \cite{White:size-singular}, almost every time is a smooth time for the individual flows. For every $t \in [t_{k},T_{k}]$ there is a unique large component $\Gamma_k(t)$ of $\Omega_{k}(t)$ by Corollary \ref{coro:limit-rescaled-MCF}. The boundary $\Sigma_k(t)$ of $\Gamma_k(t)$ is smooth and close to a Euclidean sphere with radius $(\rho_k^2 - 4t)^{1/2}$ and center $\rho_k \xi$ in the chart at infinity \eqref{eqn:chartatinfinity}. Moreover, as $k \to \infty$,
\begin{align*}
\area( (\partial^{*}\Omega_{k}(t))\setminus \Sigma_{k}(t)) & = o(\rho_{k}^{2})\\
\vol(\Omega_{k}(t)\setminus\Gamma_{k}(t)) & = o(\rho_{k}^{3}).
\end{align*}

Recall that the Hawking mass of a closed, two-sided surface $\Sigma \subset M$ is defined as
\[
m_{H}(\Sigma) = \sqrt{\frac{\area(\Sigma)}{16\pi}} \left( 1 - \frac{1}{16\pi} \int_{\Sigma} H^{2}d\mu\right).
\]

Let
\[
m_k =  \sup_{t\in[t_{k},T_{k}]} m_{H}(\Sigma_{k}(t)).
\]

\begin{corollary}\label{coro:disconn-Hawk-mass}
We have that
\[
\lim_{k \to \infty} m_k = 0.
\]
\begin{proof}
The surface  $\Sigma_{k}(t)$ is geometrically close to the coordinate sphere $S_{\sqrt {\rho_k^2 - 4 t}} (\rho_k a)$ in the chart at infinity \eqref{eqn:chartatinfinity} by Corollary \ref{coro:limit-rescaled-MCF}. The assertion follows from Appendix \ref{sec:Hawkingoutlying}.
\end{proof}
\end{corollary}

We denote by 
\[
A_m : (0, \infty) \to (0, \infty)
\]
the isoperimetric profile of Schwarzschild with mass $m > 0$. Thus, given $V > 0$, 
\[
A_m(V) = \left(1 + \frac{m}{2 r}\right)^4 4 \pi r^2
\]
where $r = r(V) > m/2$ is such that
\[
V = 4 \pi \int_{\frac{m}{2}}^r \left(1 + \frac{m}{2 r}\right)^6  r^2 d r
\]
We denote by
\[
V_m : (0, \infty) \to (0, \infty)
\]
the inverse of this function. We recall the following expansion obtained from a straightforward computation in view of H. Bray's characterization of isoperimetric surfaces in Schwarzschild  as centered coordinate spheres \cite[Theorem 8]{Bray:1997}. The claim that the error term is uniformly bounded is proven in Lemma 10 of \cite{Jauregui-Lee:2016}.
\begin{lemma}
We have that 
\[
V_{m}(A) = \frac{1}{6\sqrt{\pi}} A^{\frac{3}{2}} + \frac{m}{2} A + O(A^{\frac{1}{2}})
\]
as $A \to \infty$. The error  is uniform with respect to the parameter $m$ from a given range $0 < m \leq m_{0}$.
\end{lemma}

G. Huisken has shown \cite{Huisken:iso-mass,Huisken:MM-iso-mass-video} that the quantity  
\begin{align}\label{eq:Huisken-iso-mon}
t \mapsto - \vol(\Omega_{t}) + V_{m}(\area (\Sigma_{t})) 
\end{align}
is non-increasing along a classical mean curvature flow of boundaries 
\[
\{ \Sigma_{t} = \partial\Omega_{t} \} _{t \in (a, b)}
\]
provided that $m \geq m_{H}(\Sigma_{t})$ and $|\Sigma_{t}| > 16\pi m^{2}$ for all $t \in (a, b)$. 

J. Jauregui and D. Lee \cite{Jauregui-Lee:2016} have introduced a modification of the level set flow starting from a mean convex region along which G. Huisken's monotonicity holds. Their result applies beautifully to our setting. 

For $\Omega$ the (unique) large component of a large isoperimetric region in $(M,g)$, we consider the  \emph{modified level set flow} $\{ \hat \Omega(t)\}_{t \geq 0}$ with $\hat \Omega (0) = \Omega$ defined by J. Jauregui and D. Lee in Definitions 24 and 27 of \cite{Jauregui-Lee:2016}. The modified flow agrees with the original level set flow $\{\Omega(t)\}_{t \geq 0}$ except that components of the original flow are frozen when their perimeter drops below $36\pi (m_{ADM})^{2}$. J. Jauregui and D. Lee have shown in Proposition 30 of \cite{Jauregui-Lee:2016} that G. Huisken's monotonicity holds along their modified level set flow.  In the statement of their result below, $T \geq 0$ as in Lemma 29 of \cite{Jauregui-Lee:2016} is the time when the flow has frozen up completely. 

\begin{proposition} [\cite{Jauregui-Lee:2016}] \label{prop:weak-Huisken-mon} 
The quantity 
\[
t \mapsto   - \vol ( \hat\Omega(t)) + V_{m_{ADM}}( \area ( \partial^{*}\hat\Omega(t) ))
\]
is non-increasing on $[0, T]$.
\end{proposition}

We return to our previous setting, where each $\Omega_{k}$ is the large component of a large isoperimetric regions and where the rescaled regions $\tilde \Omega_{k}$ converge to $B_{1}(\xi)$ for some $\xi \not = 0$. We have already seen that the original level set flow $\{ \Omega_{k}(t) \}_{t \geq 0}$ with initial condition $\Omega_k(0) =\Omega_{k}$ has the property that --- for $t \in [t_k,T_{k}]$ --- there is a \emph{unique} large component $\Gamma_k(t)$ of $\Omega_k (t)$. The boundary $\Sigma_{k}(t) = \partial\Gamma_{k}(t)$ of this component is smooth. We recall that 
\[
t_{k} = \rho_{k}^{2}(\varepsilon + \tilde T(|\xi|)) \qquad \text{ and } \qquad T_{k} =\rho_{k}^{2}(1/4-\varepsilon)
\]
have been chosen as smooth times for the level set flow $\{\Omega_{k}(t)\}_{t\geq0}$. The surface $\Sigma_k(t)$ is close to a Euclidean sphere of radius $  (\rho_k^2 - 4 t)^{1/2}$  with center at $\rho_k \, \xi$ in the chart at infinity \eqref{eqn:chartatinfinity}. Consider the modified flow $\{ \hat \Omega_{k}(t)\}_{t \geq 0}$ of J. Jauregui and D. Lee described above. By what we have just said, 
\[
\area (\Sigma_k(t_k)) \geq 36\pi (m_{ADM})^2
\]
provided that $k$ is sufficiently large. We see that the large components $\Gamma_k(t_k)$ are not affected by the freezing that defines the passing from the original to the modified level set flow --- their perimeter is too large. Thus $\hat \Omega_{k}(t_k)$ is the disjoint union $E_{k} (t_k) \cup \Gamma_{k}(t_k)$ where
\begin{align} \label {eqn:estimateEj}
\vol(E_{k} (t_k))  = o(\rho_{k}^{3}) \qquad \text{ and } \qquad  \area(\partial E_{k} (t_k))  = o(\rho_{k}^{2}).
\end{align}


\section {Proof of Theorem \ref{thm:main} when $\tau > 1/2$}

We continue with the notation of Section \ref{sec:MCF}. The strategy of the proof is illustrated in Figure \ref{fig:mcf}.

\begin{proposition}\label{prop:centering-iso}
$\xi = 0$.
\end{proposition}
\begin{proof}
Assume that $\xi \neq 0$. We continue with the notation set forth above. Note that 
\[
\{ \Sigma_{k}(t) = \partial \Gamma_k(t) \}_{t \in [t_k, T_k]}
\]
is a smooth mean curvature flow. In Corollary \ref{coro:disconn-Hawk-mass} we have seen that the Hawking masses of the surfaces along this flow are bounded by $m_k = o(1)$ as $k \to \infty$. By G. Huisken's monotonicity \eqref{eq:Huisken-iso-mon} for $\Sigma_{k}(t)$ applied with the Hawking mass bound $m = m_k = o(1)$, we have that
\[
- \vol(\Gamma_{k}(t_{k})) + \frac{1}{6\sqrt{\pi}} \area(\Sigma_{k}(t_{k}))^{3/2 } + o(\rho_{k}^{2}) \geq - \vol(\Gamma_{k}(T_{k})) + \frac{1}{6\sqrt{\pi}} \area(\Sigma_{k}(T_{k}))^{3/2 } + o(\rho_{k}^{2})
\]
where we have also used that 
\[
\area(\Sigma_{k}(T_{k})) = 4\varepsilon \rho_{k}^{2} + o(\rho_{k}^{2}) \geq 36 \pi (m_k)^2
\]
as $k \to \infty$. On the other hand, by the sharp isoperimetric inequality \eqref{eqn:isoperimetricinequality} for $(M, g)$, 
\begin{align*}
- \vol(\Gamma_{k}(T_{k})) + \frac{1}{6\sqrt{\pi}} \area(\Sigma_{k}(T_{k}))^{3/2 } &\geq -m_{ADM} \area(\Sigma_{k}(T_{k}))  \\ & = -16\pi\varepsilon m_{ADM} \rho_{k}^{2} + o(\rho_{k}^{2}).
\end{align*}
Combining these two estimates, we obtain  
\begin{align}\label{eq:Gammaj-Eucl-iso}
- \vol(\Gamma_{k}(t_{k})) + \frac{1}{6\sqrt{\pi}} \area(\Sigma_{k}(t_{k}))^{3/2 } \geq -16\pi\varepsilon m_{ADM} \rho_{k}^{2} + o(\rho_{k}^{2}).
\end{align}

We now apply Proposition \ref{prop:weak-Huisken-mon} to the modified weak flow $\{ \hat\Omega_{k}(t) \}_{t \geq 0}$ between the (smooth) times $t=0$ and $t= t_{k}$. In the first line below we use that $\Omega_k$ --- as the substantial component of a large isoperimetric region --- almost saturates the sharp isoperimetric inequality \eqref{eqn:isoperimetricinequality} on $(M, g)$.   
\begin{align*}
0  = 
&-\vol(\Omega_{k}) + \frac{1}{6\sqrt{\pi}} \area(\partial \Omega_{k})^{3/2 } + \frac{m_{ADM}}{2} \area(\partial \Omega_{k}) + o (\rho_k^2) \\
 \geq &- \vol(\hat \Omega_{k}(t_{k})) + \frac{1}{6\sqrt{\pi}} \area(\partial \hat \Omega_{k}(t_{k}))^{3/2 } + \frac {m_{ADM}}{2} \area(\partial \hat \Omega_{k}(t_{k})) + o(\rho_{k}^{2})\\
 =& - \vol(\Gamma_{k}(t_{k})) + \frac{1}{6\sqrt{\pi}} \area(\Sigma_{k}(t_{k}))^{3/2 } \\
&\qquad   -  \vol(E_{k} (t_k)) + \frac{1}{6\sqrt{\pi}} \left( (\area(\Sigma_{k}(t_{k})) + \area(\partial E_{k} (t_k)) )^{3/2 } - \area(\Sigma_{k}(t_{k}))^{3/2 }\right) \\
& \qquad + \frac{m_{ADM}}{2} \area(  \Sigma_{k}(t_{k})) + o(\rho_{k}^{2}) \\
\geq   & -  \vol(E_{k} (t_k)) + \frac{1}{6\sqrt{\pi}} \left( (\area(\Sigma_{k}(t_{k})) + \area(\partial E_{k} (t_k)) )^{3/2 } - \area(\Sigma_{k}(t_{k}))^{3/2 }\right) \\
& \qquad + \frac{m_{ADM}}{2} \area(\Sigma_{k}(t_{k})) - 16\pi \varepsilon m_{ADM}\rho_{k}^{2} + o(\rho_{k}^{2}) .
\end{align*}
The final inequality follows from \eqref{eqn:estimateEj} and \eqref{eq:Gammaj-Eucl-iso}. 

Assume first that $\area(\partial E_{k} (t_k)) = O(1)$ as $j\to\infty$. Then $\vol(E_{k} (t_k)) = O(1)$ as well, and 
\[
 -  \vol(E_{k} (t_k)) + \frac{1}{6\sqrt{\pi}} \left( (\area(\Sigma_{k}(t_{k})) + \area(\partial E_{k} (t_k)) )^{3/2 } - \area(\Sigma_{k}(t_{k}))^{3/2 }\right)  \geq - O(1)
\] 
as $j\to\infty$. Thus
\begin{align} \label{eqn:auxcontra}
\area(\Sigma_{k}(t_{k})) \leq  (8 \varepsilon + o (1) ) 4 \pi  \rho_{k}^{2}.
\end{align}
This contradicts the choice $\varepsilon > 0$ in \eqref{eqn:choiceepsilon}, because
\[
\area(\Sigma_{k}(t_{k})) = (1- 4 \varepsilon - 4\tilde T(|\xi|) + o (1)) 4 \pi \rho_{k}^{2}.
\]

Assume now that $\area(\partial E_{k} (t_k)) \to\infty$ as $k \to \infty$. Then 
\[
\vol(E_{k} (t_k)) \leq \frac{1}{6\sqrt{\pi}} \area(\partial E_{k} (t_k))^{3/2 } + m_{ADM} \area (\partial E_k(t_k))
\]
by the sharp isoperimetric inequality \eqref{eqn:isoperimetricinequality}. Combining this with the above and \eqref{eqn:estimateEj}, we have
\begin{align*}
 0 \, \geq  \, & \frac{1}{6\sqrt{\pi}} \left( \left (\area(\Sigma_{k}(t_{k})) + \area(\partial E_{k} (t_k)) \right)^{3/2 } - \area(\Sigma_{k}(t_{k}))^{3/2 } - \area(\partial E_{k} (t_k))^{3/2 } \right) \\
& + \frac {m_{ADM}}{2} \area(\Sigma_{k}(t_{k})) - 16\pi \varepsilon m_{ADM} \rho_{k}^{2} + o(\rho_{k}^{2}).
\end{align*}
Using that
\[
x^{3/2 } + y^{3/2 } \leq (x+y)^{3/2 } 
\]
for all $x, y \geq 0$, we arrive again at the contradictory estimate  \eqref{eqn:auxcontra}.
\end{proof}

\begin {proof} [Proof of Theorem \ref{thm:main}] Combining Lemma  \ref{lem:decomposition} and Proposition \ref{prop:centering-iso}, we see that every sufficiently large isoperimetric region is connected and close to the centered coodinate ball $B_1(0)$ when put to scale of its volume in the chart at infinity \eqref{eqn:chartatinfinity}. By the uniqueness of large stable constant mean curvature spheres described in Appendix \ref{sec:canonicalfoliation}, the outer boundary of such an isoperimetric region is a leaf of the canonical foliation. 
\end {proof}

\begin{figure} 
\begin{tikzpicture}[scale=.8] 
	\begin{scope}[shift = {(5.5,8.5)}]
		\node at (50:4.5) {$\mathbb{R}^{3}\setminus\{0\}$};

		\filldraw (0,0) circle (1pt);
		\node at (.2,0) {$\scriptstyle{\xi}$};
		
		\draw [->]  (80:.05) -- (80:2.95);
		\node at (140: 3.7) {$B_{1}(\xi)$};

		\draw [dashed] (0,0) circle (3);
		\node at (0:3.6) {$\partial\tilde \Omega_{k}$};
		\draw [smooth cycle, tension = .7] plot coordinates {(0:3) (45:3.2) (90:3.1) (135:2.9) (180:3.1) (225:3.05) (270:2.92) (315:3.2)};
		
		\node at (270:2.4) {$\scriptstyle{\times}$}; 
		
		\node at (270:4) {$(a)$};
	\end{scope}

	\node at (50:4.5) {$\mathbb{R}^{3}\setminus\{0\}$};

	\filldraw [opacity = .2, smooth cycle, tension = .7] plot coordinates {(0:3) (45:3.2) (90:3.1) (135:2.9) (180:3.1) (225:3.05) (270:2.92) (315:3.2)};
	\filldraw [white, smooth cycle, tension = .7] plot coordinates {(0:2.05) (45:2.07) (90:2) (135:1.92) (180:1.93) (225:2) (270:2.05)  (315:2)};
	\filldraw [white, smooth cycle, tension = .7] plot coordinates {(0:.5) (45:.53) (90:.5) (135:.48) (180:.5) (225:.52) (270:.5)  (315:.48)};
	\filldraw [white, smooth cycle] plot coordinates {(270:2.6) (274:2.3) (270:2.2) (263:2.3)};

	\node at (0:3.6) {$\partial\tilde \Omega_{k}$};
	\draw [smooth cycle, tension = .7] plot coordinates {(0:3) (45:3.2) (90:3.1) (135:2.9) (180:3.1) (225:3.05) (270:2.92) (315:3.2)};
	
	\node at (270:2.4) {$\scriptstyle{\times}$}; 
	
	\filldraw (0,0) circle (1pt);
	\node at (.2,0) {$\scriptstyle{\xi}$};

	\draw [smooth cycle, tension = .7] plot coordinates {(0:2.5) (45:2.6) (90:2.4) (135:2.4) (180:2.4) (225:2.5) (260:2.5) (270:2.75) (280:2.5) (315:2.6)};
	\node at (200: 4.2) {$\partial\tilde E_{k}(t_{k})$};
	\draw [->] (205:3.4) -- (261:2.3);
	\draw [smooth cycle] plot coordinates {(270:2.6) (274:2.3) (270:2.2) (263:2.3)};

	\node at (160:4) {${\tilde \Sigma_{k}(t_{k})}$};
	\draw [->] (157:3.3) -- (160:2);
	\draw [smooth cycle, tension = .7] plot coordinates {(0:2.05) (45:2.07) (90:2) (135:1.92) (180:1.93) (225:2) (270:2.05)  (315:2)};

	\node at (120:4) {${\tilde \Sigma_{k}(T_{k})}$};
	\draw [->] (120:3.5) -- (120:.55);
	\draw [smooth cycle, tension = .7] plot coordinates {(0:.5) (45:.53) (90:.5) (135:.48) (180:.5) (225:.52) (270:.5)  (315:.48)};

		\node at (270:4) {$(b)$};

\begin{scope}[shift = {(11,0)}] 
	
	\node at (50:4.5) {$\mathbb{R}^{3}\setminus\{0\}$};

	\filldraw [white, smooth cycle, tension = .7] plot coordinates {(0:2.05) (45:2.07) (90:2) (135:1.92) (180:1.93) (225:2) (270:2.05)  (315:2)};
	\filldraw [opacity=.1,smooth cycle, tension = .7] plot coordinates {(0:2.05) (45:2.07) (90:2) (135:1.92) (180:1.93) (225:2) (270:2.05)  (315:2)};
	\filldraw [white, smooth cycle, tension = .7] plot coordinates {(0:.5) (45:.53) (90:.5) (135:.48) (180:.5) (225:.52) (270:.5)  (315:.48)};
	\filldraw [white, smooth cycle] plot coordinates {(270:2.6) (274:2.3) (270:2.2) (263:2.3)};

	\node at (0:3.6) {$\partial\tilde \Omega_{k}$};
	\draw [smooth cycle, tension = .7] plot coordinates {(0:3) (45:3.2) (90:3.1) (135:2.9) (180:3.1) (225:3.05) (270:2.92) (315:3.2)};
	
	\node at (270:2.4) {$\scriptstyle{\times}$}; 
	
	\filldraw (0,0) circle (1pt);
	\node at (.2,0) {$\scriptstyle{\xi}$};

	\draw [smooth cycle, tension = .7] plot coordinates {(0:2.5) (45:2.6) (90:2.4) (135:2.4) (180:2.4) (225:2.5) (260:2.5) (270:2.75) (280:2.5) (315:2.6)};
	\node at (200: 4.2) {$\partial\tilde E_{k}(t_{k})$};
	\draw [->] (205:3.4) -- (261:2.3);
	\draw [smooth cycle] plot coordinates {(270:2.6) (274:2.3) (270:2.2) (263:2.3)};

	\node at (160:4) {${\tilde \Sigma_{k}(t_{k})}$};
	\draw [->] (157:3.3) -- (160:2);
	\draw [smooth cycle, tension = .7] plot coordinates {(0:2.05) (45:2.07) (90:2) (135:1.92) (180:1.93) (225:2) (270:2.05)  (315:2)};

	\node at (120:4) {${\tilde \Sigma_{k}(T_{k})}$};
	\draw [->] (120:3.5) -- (120:.55);
	\draw [smooth cycle, tension = .7] plot coordinates {(0:.5) (45:.53) (90:.5) (135:.48) (180:.5) (225:.52) (270:.5)  (315:.48)};
	\node at (270:4) {$(c)$};
\end{scope}
\end{tikzpicture}

\caption{
We depict here the case where $0<|\xi| < 1$. In (a), a sequence of large isoperimetric regions $\Omega_{k}$ is assumed to limit to $B_{1}(\xi)$ after rescaling. The convergence is smooth on compact subsets of $\mathbb{R}^{3}\setminus\{0\}$. Here, the origin is denoted by ``$\times$.''  
\\ \ \\ 
In (b) and (c), we depict boundaries of the (modified) level set flows. We show that the large component of the level set flow ``disconnects.'' The large disconnected component is labeled $\tilde\Sigma_{k}(t_{k})$. It is possible that there are additional components $\tilde E_{k}(t_{k})$ of the modified flow. 
\\ \ \\  
In (b), the change of the isoperimetric deficit as the flow sweeps out the shaded region is estimated by the Hawking mass bound of $m_{ADM}$. On the other hand, in (c), the lightly shaded region is swept out by surfaces with Hawking mass bounded by $o(1)$ as $k\to\infty$. This leads to improved estimates for the deficit, showing that the original region $\Omega_{k}$ cannot have been isoperimetric. 
\\ \ \\ 
When $|\xi|=1$, a similar situation occurs, except the flow disconnects from the origin after a short time (in the rescaled picture). We must wait this short time before arguing as in (c), so there will be a thin region as in (b) in this case. If $|\xi|>1$, the flow is completely disconnected, so we do not need to consider the shaded region as in (b).
} 
\label{fig:mcf}
\end{figure}
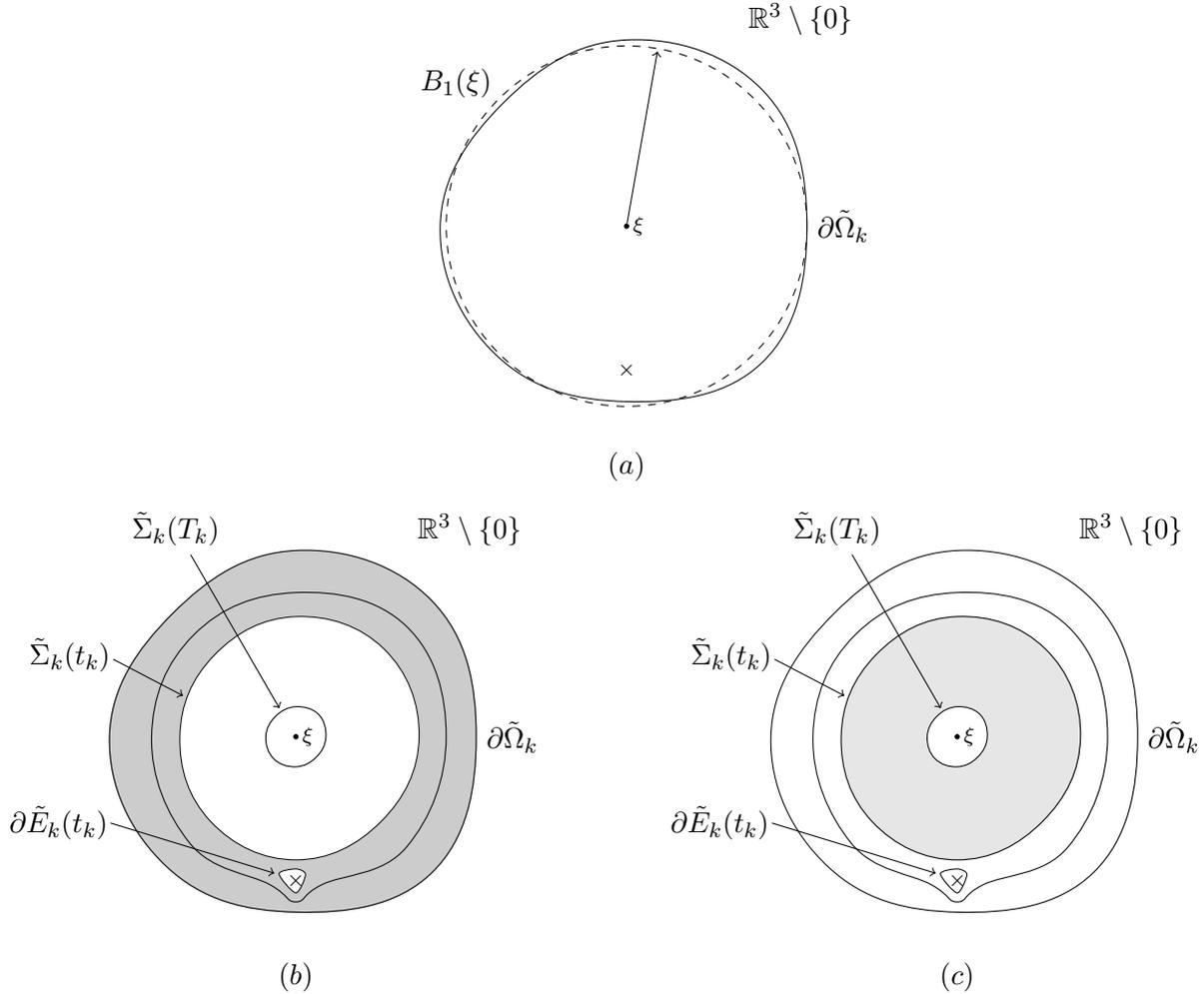


\appendix 


\section {Canonical foliation} \label{sec:canonicalfoliation}

In this section, we state results on the existence and uniqueness of a canonical foliation through stable constant mean curvature spheres of the end of an asymptotically flat Riemannian $3$-manifold $(M, g)$ with positive mass. The generality of the discussion here is tailored to our application in the proof of Theorem \ref{thm:main}. In particular, the assumption of non-negative scalar curvature can be replaced by a stronger decay assumption on the scalar curvature; see the work of C. Nerz  \cite{Nerz:2014}. \\

All results discussed below depart from the pioneering work of G. Huisken and S.-T. Yau \cite{Huisken-Yau:1996} and of J. Qing and G. Tian \cite{Qing-Tian:2007} for initial data that is \emph{asymptotic to Schwarzschild} with positive mass. We also mention here the crucial intermediate results of L.-H. Huang \cite{Huang:2010} for \emph{asymptotically even data}. We refer to the recent articles \cite{Cederbaum-Nerz:2015} by C. Cederbaum and C. Nerz, \cite{Huang:2011} by L.-H. Huang, \cite{Ma:2016}  by S. Ma, and \cite{Nerz:2014} by C. Nerz for an overview of the literature on this exceptionally rich subject.  \\

The following uniqueness and existence results are, in the stated generality, due to C. Nerz \cite{Nerz:2014}. Let $(M, g)$ be a Riemannian $3$-manifold that is asymptotically flat at rate $\tau > 1/2$ and which has non-negative scalar curvature and $m_{ADM} > 0$. There are $H_0 > 0$, a compact subset $C \subset M$ with $B_1 \subset C$, and a diffeomorphism 
\[
\Phi : (0, H_0) \times \mathbb{S}^2 \to M \setminus C
\]
such that 
\[
\Phi (\{H\} \times \mathbb{S}^2) = \Sigma^H
\]
is a constant mean curvature sphere with mean curvature $H>0$ for every $H \in (0, H_0)$. In the chart at infinity \eqref{eqn:chartatinfinity}, 
\[
(H/2) \, \Sigma^H \to S_1(0) = \{ x \in \R^3: |x| = 1\}
\]
smoothly as $H \searrow 0$. We have, by the remark preceding Proposition A.1 in \cite{Nerz:2014}, that
\begin{align} \label{eqn:ADMHawking}
m_{ADM} = \lim_{H \searrow 0} \, \sqrt { \frac{\area (\Sigma^H)}{16 \pi}} \left( 1 - \frac{H^2 \area (\Sigma^H)}{16 \pi}\right).
\end{align}
Moreover, $\Sigma^H$ is the unique stable constant mean curvature sphere of mean curvature $H$ that is geometrically close to the coordinate sphere $S_{2/H} (0)$ in the chart at infinity \eqref{eqn:chartatinfinity}. \\

S. Ma  has shown in \cite{Ma:2016} that under the stronger decay assumption that 
\begin{align} \label{eqn:Masigma}
|x|^{|\alpha|} |(\partial^\alpha \sigma_{ij}) (x)| = O (|x|^{-1}) \qquad \text{ as } \qquad |x| \to \infty
\end{align}
for all multi-indices $\alpha$ of length $|\alpha| = 0, 1, 2, 3, 4$  (one additional derivative) and
\begin{align} \label{eqn:MaR}
R(x) = O (|x|^{- 3 - \epsilon})\qquad \text{ as } \qquad |x| \to \infty
\end{align}
for some $\epsilon > 0$ in the chart at infinity \eqref{eqn:chartatinfinity}, the compact subset $C \subset M$ above can be chosen so that each leaf $\Sigma^H$ of the canonical foliation is the only stable constant mean curvature sphere of mean curvature $H \in (0, H_0)$ enclosing $C$.


\section {General properties of the isoperimetric profile}

In this section, we recall some useful properties about the \emph{isoperimetric profile} \eqref{eqn:isoperimetricprofile} 
\[
A : (0, \infty) \to (0, \infty)
\]
of an asymptotically flat Riemannian $3$-manifolds $(M, g)$ that are used throughout the paper. The general results about the isoperimetric profile discussed below are established in e.g. \cite{Bavard-Pansu:1986, Bray:1997, Ros:2005, Flores-Nardulli:2014}.\\
 
Locally, the isoperimetric profile can be written as the sum of a concave and a smooth function. In particular, the isoperimetric profile is absolutely continuous and its left and right derivatives $A^-(V)$, $A^+(V)$ exist at every $V > 0$, and they agree at all but possibly countably many $V >0$. We have that 
\[
 \lim_{W \searrow V} A'{}^+ (W) \leq A'{}^+ (V) \leq A'^- (V) \leq \lim_{W \nearrow V} A'{}^- (W). 
\]
Assume that for some $V>0$ there is $\Omega_V \in \mathcal{R}_V$ with 
\[
A(V) = \area (\partial \Omega_V) - \area (\partial M).
\]
Such \emph{isoperimetric regions} exist for every sufficiently large volume $V>0$ when the mass of $(M, g)$ is positive by (the proof of) Theorem 1.2 in \cite{hdiso}, and for every volume $V>0$ when the scalar curvature of $(M, g)$ is non-negative by Proposition K. 1 in \cite{mineffectivePMT}.  The boundary $ \partial \Omega_V \setminus \partial M$ is a stable constant mean curvature surface. Its mean curvature $H$ is positive when computed with respect to the outward unit normal. Moreover, 
\[
A'^{+} (V) \leq H \leq A'^{-} (V).
\]
In particular, the isoperimetric profile is a strictly increasing function. At volumes $V > 0$ where the isoperimetric profile is differentiable, the boundaries of all isoperimetric regions of volume $V$ have the same constant mean curvature. 


\section{Sharp isoperimetric inequality} \label{sec:sharpisoperimetric}

The  characterization of the ADM-mass through the isoperimetric deficit of large centered coordinate spheres in Lemma \ref{lem:FSTM} below was proposed by G. Huisken \cite{Huisken:iso-mass} and proved by X.-Q. Fan, P. Miao, Y. Shi, and L.-F. Tam as Corollary 2.3 in \cite{Fan-Shi-Tam:2009}. 

\begin {lemma} \label{lem:FSTM}
Let $(M, g)$ be a complete Riemannian $3$-manifold that is asymptotically flat. Then 
\[
m_{ADM} = \lim_{\rho \to \infty} \frac{2}{\area(S_\rho)} \left( \vol (B_\rho) - \frac{\area(S_\rho)^{3/2}}{6 \sqrt \pi}\right).
\]
\end {lemma}

The following result was proposed by G. Huisken \cite{Huisken:iso-mass, Huisken:MM-iso-mass-video} and proven in detail by J. Jauregui and D. Lee as Theorem 3 in \cite{Jauregui-Lee:2016}. 

\begin {theorem} Let $(M, g)$ be a complete Riemannian $3$-manifold with non-negative scalar curvature that is asymptotically flat of rate $\tau > 1/2$. Then 
\[
m_{ADM} (M, g) = m_{iso} (M, g)
\]
where 
\[
m_{iso} (M, g) = \sup_{\{\Omega_i\}_{i=1}^\infty} \limsup_{i \to \infty} \frac{2}{\area (\partial \Omega_i)} \left( \vol (\Omega_i) - \frac{\area(\partial \Omega_i)^{3/2}}{6 \sqrt \pi }\right).  
\]
The supremum here is taken over all sequences $\{\Omega_i\}_{i=1}^\infty$ of smooth compact outward area-minimizing  regions that are increasing to $M$. 
\end {theorem}

We recall from \cite{Jauregui-Lee:2016} that the inequality
\[
m_{ADM} (M, g) \leq m_{iso} (M, g) 
\] 
follows from Lemma \ref{lem:FSTM}. Note that 
\[
m_{iso} (M, g) \leq \limsup_{V \to \infty}  \frac{2}{A(V)} \left(V - \frac{A(V)^{3/2 }}{6\sqrt{\pi}} \right).
\] 
We present below a short new proof of the reverse inequality that is based on the behavior of large isoperimetric regions. 

\begin{theorem} \label{prop:huisk-iso-mass}
Let $(M, g)$ be a complete Riemannian $3$-manifold with non-negative scalar curvature that is asymptotically flat of rate $\tau > 1/2$. Then 
\[
m_{ADM} = \lim_{V\to\infty} \frac{2}{A(V)} \left(V - \frac{A(V)^{3/2 }}{6\sqrt{\pi}} \right). 
\]

\begin{proof} Lemma \ref{lem:FSTM} implies that 
\[
m_{ADM} \leq \liminf_{V\to\infty} \frac{2}{A(V)} \left(V - \frac{A(V)^{3/2 }}{6\sqrt{\pi}} \right). 
\]
Indeed, for every $V > 0$, the function 
\[
x \mapsto \frac{2}{x} \left( V - \frac{x^{3/2}}{6 \sqrt \pi}\right)
\]
is decreasing on $(0, \infty)$. 

For the proof of the reverse inequality, assume first that $\partial M = \emptyset$. 

Let $V > 0$ large be such that $A'(V)$ exists. An isoperimetric region $\Omega$ of volume $V$ is connected with connected, outward area-minimizing boundary $\Sigma$ of constant mean curvature $A'(V) = H > 0$. Using the work of G. Huisken and T. Ilmanen \cite{Huisken-Ilmanen:2001} as stated in Lemma \ref{lem:monotonicity}, we see that
\[
\sqrt {\frac{A(V)}{16 \pi}} \left (1 - \frac{1}{16 \pi} A' (V)^2 A(V)\right) = \sqrt {\frac{\area(\Sigma)}{16 \pi}} \left( 1 - \frac{1}{16\pi} H^{2} \area(\Sigma) \right) \leq  m_{ADM}. 
\]
From this, we compute that 
\begin{align*}
\left( V - \frac{A(V)^{3/2 }}{6\sqrt{\pi}} \right)' & = 1 - \frac{1}{4\sqrt{\pi}}A'(V)A(V)^{1/2 }\\
& = \frac{1- \frac{1}{16\pi} A'(V)^{2}A(V)}{1 + \frac{1}{4\sqrt{\pi}}A'(V)A(V)^{1/2 }}
 \leq \frac{4\sqrt{\pi}  A(V)^{-1/2 }}{1 + \frac{1}{4\sqrt{\pi}}A'(V)A(V)^{1/2 }} m_{ADM}
\end{align*}
Using the remarks following Lemma \ref{lem:decomposition}, we see that $A'(V)\sqrt{A(V)}$ approaches $4\sqrt{\pi} $ as $V\to\infty$. It follows that the above expression is bounded above by 
\[
\frac{1}{2}A'(V) m_{ADM} (1 + o(1))
\]
as $V \to \infty$. Using that the isoperimetric profile is absolutely continuous, it follows that 
\[
\limsup_{V \to \infty}  \frac{2}{A(V)} \left(V - \frac{A(V)^{3/2 }}{6\sqrt{\pi}} \right) \leq m_{ADM}.
\]
In the general case, where $\partial M \neq \emptyset$, we work with the (unique) large component of a large isoperimetric region instead. The omit the formal modifications of the proof. 
\end{proof}
\end{theorem}

\begin {corollary} [Sharp isoperimetric inequality] \label{cor:isoperimetricinequality} Let $(M,g)$ be an asymptotically flat Riemannian $3$-manifold with non-negative scalar curvature. Let $\Omega \subset M$ be a compact region. Then 
\begin{align} \label{eqn:isoperimetricinequality}
\vol (\Omega) \leq \frac{\area (\partial \Omega)^{3/2}}{6 \sqrt \pi} +  \frac{m_{ADM}}{2} \area (\partial \Omega) + o(1) \area (\partial \Omega)
\end{align}
as $\vol (\Omega) \to \infty$.
\end {corollary}


\section {Extension of a Brakke flow across a point}

In this section, we follow the notation, the conventions, and some of the ideas in T. Ilmanen's article \cite{Ilmanen:ellipticreg} closely.  \\

We define an injective map of Radon measures
\[
\hat {\mathcal{M}}(\R^{n+1}\setminus\{0\}) := \{\mu \in \mathcal{M}(\R^{n+1}\setminus \{0\}) : \mu(B_{1} (0)\setminus\{0\}) < \infty\} \to \mathcal{M}(\R^{n+1})
\]
that extends  $\mu \in \hat {\mathcal{M}}(\R^{n+1}\setminus\{0\})$ to a Radon measure $\hat \mu \in \mathcal{M}(\R^{n+1})$ such that 
\[
\hat \mu (\{0\}) = 0.
\]
 This map restricts to an injection of integer $n$-rectifiable Radon measures 
 \[
 \hat{\mathcal{M}}(\R^{n+1}\setminus\{0\})\cap \mathcal{IM}_{n}(\R^{n+1}\setminus\{0\})\to \mathcal{IM}_{n}(\R^{n+1})
 \]
 which in turn lifts to an injection of integer $n$-rectifiable varifolds  
 \[
\{ V \in \mathbf{IV}_{n}(\R^{n+1}) : \mu_V (B_1(0) \setminus \{0\}) < \infty \}  \to \mathbf{IV}_{n}(\R^{n+1})
 \]
 which we denote by 
 \[
 V \mapsto \hat V.
 \]
 
The extension of a stationary varifold across a point is not necessarily again stationary. 

\begin{example} \label{example:rays}
Let $\theta_1, \ldots, \theta_m \in \R$. Consider the rays $\ell_k = [0, \infty) \, e^{i \theta_k} \subset \R^2$. 
The varifold $V=\bigcup_{k=1}^{m}|  \ell_k|$ is stationary as an element of $\mathcal{IM}_{1}(\R^{2}\setminus\{0\})$. It is stationary as an element of $\mathcal{IM}_{1}(\R^{2})$ if and only if  $e^{i\theta_1} + \ldots + e^{i \theta_m} = 0$.
\end{example}

This phenomenon in the previous example is particular to dimension $n =1$. 

\begin {lemma} Let $n \geq 2$. There are radial functions $\chi_k \in C_c^\infty(B_1 (0))$ with $0 \leq \chi_k \leq 1$ such that $\chi_k (x) = 1$ when $|x| < 1/(2k^2)$ and $\chi_k  (x)= 0$ when $|x| > 1/k$ and constants $c_k \searrow 0$ with the following property. Let $\mu$ be a measure on $B_1(0) \setminus \{0\}$ such that, for come $c > 0$, 
\begin{align} \label{eqn:monotonicity}
\mu (B_\rho (0) \setminus \{0\}) \leq c \rho^n
\end{align}
for all $0 < \rho \leq 1$. Then 
\[
\frac{1}{c} \int|\nabla \chi_k|^2 d \mu \leq  c_k.
\]
\end {lemma}

Below, we will often work with the functions 
\begin{align} \label{eqn:cutoff}
\varphi_k = 1 - \chi_k \in C^\infty(\R^{n+1} \setminus \{0\}).
\end{align}
Note that $0 \leq \varphi_k \to 1$ locally uniformly on $\R^{n+1} \setminus \{0\}$ and that, under the assumptions of the previous lemma,
\[
\lim_{k \to \infty} \int |\nabla \varphi_k|^2 d \mu = 0.
\]

We include a proof of the following, well-known result as preparation for Proposition \ref{prop:extend-brakke-pt}.

\begin{lemma} [Extending stationary varifolds across a point] \label {lem:extendvarifold}
Let $n\geq 2$. Let $V$ be a stationary $n$-rectifable varifold on $\R^{n+1}\setminus\{0\}$ such that $\mu_V (B_{1} (0)\setminus\{0\}) < \infty$. The extension $\hat V$ of $V$ across the origin is stationary as an $n$-rectifiable varifold on $\R^{n+1}$.

\begin{proof}
Let $\varphi_{k} \in C^\infty(\R^{n+1} \setminus \{0\})$ be cut-off functions as in \eqref{eqn:cutoff}. Note that \eqref{eqn:monotonicity} holds by the monotonicity formula for stationary varifolds as stated in (17.5) of \cite{Simon:GMT}.
Let $X \in C^{1}_{c}(\R^{n+1};\R^{n+1})$. Then 
\[
0 = \int_{\Sigma}\text{div}_{\Sigma}(\varphi_{k}X) d\mu_{V} = \int_{\Sigma} \varphi_{k} \text{div}_{\Sigma}X\, d\mu_{V}+ \int_{\Sigma} X\cdot \proj_{T \Sigma} \nabla \varphi_{k} \,d\mu_{V}
\]
because $V = \underline{v}(\Sigma,\theta)$ is stationary in $\R^n\setminus \{0\}$. As $k \to \infty$, the first term on the right tends to $(\delta \hat V )(X)$, while the second term tends to $0$. 
\end{proof}
\end{lemma}

\begin{lemma}\label{lem:L2-H-extend-pt}
Let $n \geq 2$. Let $V$ be an $n$-rectifiable varifold on $\R^{n+1}\setminus \{0\}$ such that
\[
\mu_V (B_{\rho} (0) \setminus \{0\}) \leq c \rho^n
\] 
for some $c > 0$ and all $0 < \rho \leq 1$. Let $\phi \in C^{\infty}_{c}(\R^{n+1})$ be a non-negative function such that $V$ is $n$-rectifiable on $\{ x \in \R^{n+1} \setminus : \phi (x) > 0\}$ with absolutely continuous first variation such that $\int \phi |\mathbf{H}|^{2} d \mu_V < \infty$. The first variation of the extension $\hat V$ of $V$ across the origin is absolutely continuous on $\{x \in \R^{n+1}: \phi (x) > 0\}$.
\end{lemma}
\begin{proof}
Let $X \in C^{1}_{c}(\{x \in \R^{n+1} : \phi (x) > 0\}, \R^{n+1})$. Let $\varphi_{k} \in C^\infty(\R^{n+1} \setminus \{0\})$ be cut-off functions as in \eqref{eqn:cutoff}. We compute that 
\[
(\delta V)(\varphi_{k} \sqrt{\phi} X)  
= \int \varphi_{k} \sqrt{\phi} \mathbf{H} \cdot X d\mu_{V} \leq \left( \int \phi |\mathbf{H}|^{2}d\mu_{V}\right)^{\frac{1}{2}} \left( \int |X|^{2} d\mu_{V}\right)^{\frac 12} 
\leq C\Vert X\Vert_{L^{2}(\mu_{V})}
\]
and 
\[
(\delta V)(\varphi_{k}\sqrt{\phi} X) = \int\varphi_{k} \sqrt{\phi} \text{div}_{\Sigma} X \, d\mu_{V} + \int \sqrt{\phi} (\proj_{T\Sigma}\nabla \varphi_{k})\cdot X\, d\mu_{V} + \int\varphi_{k} (\proj_{T\Sigma} \nabla \sqrt{\phi}) \cdot X\, d\mu_V
\]
where $V = \underline{v}(\Sigma,\theta)$. In the last expression, the second term tends to zero by H\"older's inequality and the construction of $\varphi_{k}$, while the first term tends to $(\delta \hat V)(X)$. Finally, by H\"older's inequality, we may bound the third term by $\Vert |\nabla \phi|/\phi\Vert_{L^{2}(\mu_{V})} \Vert X\Vert_{L^{2}(\mu_{V})}$. The first quantity here can be bounded using the estimate in Lemma 6.6 of \cite{Ilmanen:ellipticreg}. Putting these facts together, we find that 
\[
|(\delta \hat V)(\sqrt \phi X)| \leq C\Vert X\Vert_{L^{2}(\mu_{ V})} = C \Vert X\Vert_{L^{2}(\mu_{\hat V})}  .
\]
This completes the proof. 
\end{proof}

We now turn to the situation for Brakke flows. 

\begin{proposition}[Extending Brakke flows across a point] \label{prop:extend-brakke-pt}  
Let $n \geq 2$. Let $\{\mu_{t}\}_{t\geq 0}$ be a codimension one integral Brakke flow on $\R^{n+1}\setminus\{0\}$ such that, for some constant $c > 0$, 
\[
 \mu_{t}(B_{\rho} (0)\setminus\{0\}) \leq c   \rho^n
\]
for all $t \geq 0$ and $0 < \rho \leq 1$. Then $\{ \hat \mu_{t}\}_{t\geq0}$ is a codimension one integral Brakke flow on $\R^{n+1}$. 

\begin {proof}

We use the cut-off functions $\varphi_k \in C^\infty(\R^{n+1} \setminus \{0\})$ from \eqref{eqn:cutoff}. Let $0 \leq \phi \in C^{2}_{c}(\R^{n+1})$.

Recall from Lemma 6.6 in \cite{Ilmanen:ellipticreg} that, on  $\{x \in \R^{n+1} : \phi(x) > 0\}$,
\begin{align} \label{eqn:Cheng}
\frac{|\nabla \phi|^2}{\phi} \leq 2 \max |\nabla^2 \phi|.
\end{align}

In a first step, we verify that, for all $t \geq 0$,
\begin{align} \label{eqn:continuityBrakkeB}
\lim_{k \to \infty} \mathcal{B} (\mu_t, \varphi_k^2 \phi) = \mathcal{B} (\hat \mu_t, \phi).
\end{align}
Assume first that $\mathcal{B}(\hat \mu_{t},\phi) > -\infty$. Then
\begin{align}
- \infty 
&< \mathcal{B} (\hat \mu_t, \varphi_k^2 \phi) = \mathcal{B}(\mu_{t},\varphi_{k}^{2}\phi) \nonumber \\ 
&= \int_{ \{ \varphi_k^2 \phi > 0\}} \left( - \varphi_{k}^{2}\phi |\mathbf{H}|^{2} +\varphi_{k}^{2} (\proj_{T^{\perp}\Sigma_{t}}\nabla \phi ) \cdot \mathbf{H}  +2\varphi_{k} \phi (\proj_{T^{\perp}\Sigma_{t}}\nabla \varphi_{k}) \cdot \mathbf{H}\right) d\mu_{t} \label {eqn:extendingbrakkecontinuity} 
\end{align}
for all $k$. The sum of the first two terms in \eqref{eqn:extendingbrakkecontinuity} tends to $\mathcal{B}(\hat \mu_{t},\phi)$ as $k \to \infty$ by dominated convergence. Using the H\"older inequality and the properties of the functions $\varphi_{k}$, we see that the third term tends to zero. This verifies \eqref{eqn:continuityBrakkeB} when $\mathcal{B} (\hat \mu_t, \phi) > - \infty$.

Assume now that $\liminf_{j\to\infty} \mathcal{B}(\mu_{t},\varphi_{k}^2 \phi) > - \infty$. From \eqref{eqn:extendingbrakkecontinuity}, we see that 
\[
\int_{\{ x \in \R^{n+1} \setminus \{0\} : \phi (x) > 0\}} \phi |\mathbf{H}|^2 d \mu_t = \limsup_{j\to\infty} \int_{\{ \varphi_k^2 \phi > 0\}} \varphi_{k}^{2} \phi|\mathbf{H}|^{2}d\mu_{t} < \infty.
\]
Lemma \ref{lem:L2-H-extend-pt} shows that $\mathcal{B}(\hat \mu_{t},\phi) > -\infty$. The claim follows from our earlier computation. 

Finally, it is easy to see that $\lim_{k \to \infty} \mathcal{B} (\mu_t, \varphi_k^2 \phi ) = - \infty$ when $\mathcal{B}(\hat \mu_t, \phi) = - \infty$.

Estimating \eqref{eqn:extendingbrakkecontinuity} as in \S 6.7 of \cite{Ilmanen:ellipticreg}, we see that 
\begin{align*}
\mathcal{B}(\mu_{t},\varphi_{k}^{2}\phi) \leq \int_{\{ \varphi_k^2 \phi > 0\}} \left( - \frac{1}{4}  \varphi_{k}^{2}\phi |\mathbf{H}|^{2} + \frac{1}{2} \varphi_{k}^{2} \frac{|\nabla \phi|^{2}}{\phi} +  4 \phi |\nabla \varphi_{k}|^{2} \right) d\mu_{t}.
\end{align*}
In combination with \eqref{eqn:Cheng} and the uniform mass bounds, we obtain
\begin{align} \label{eqn:Brakkejtbound}
\sup_{k} \sup_{t \geq 0} \mathcal{B} (\mu_t, \varphi_k^2 \phi) = C(\phi) < \infty.
\end{align}
As in \cite[\S 7.2(i)]{Ilmanen:ellipticreg},  in conjunction with the Brakke property for $\{\mu_t\}_{t \geq 0}$ this estimate implies that 
\[
t \mapsto \mu_t (\varphi_k^2 \phi) - C(\phi) t
\]
is non-increasing. Passing to the limit as $k \to \infty$ and using the uniform mass bounds, it follows that 
\[
t \mapsto \hat \mu_t (\phi) - C(\phi) t
\]
is non-increasing. 

We now verify that 
\[
\overline D_t \hat \mu_t (\phi) \leq \mathcal{B} (\hat \mu_t, \phi)
\] 
for all $t \geq 0$.  The argument follows a step of the proof of Theorem 7.1 on pp. 40--41 in  \cite {Ilmanen:ellipticreg} closely. 

Fix $t \geq 0$. We may assume that $- \infty < \overline D_{t} \hat \mu_{t}(\phi)$. Consider times $t_k \nearrow t$ and
\[
\overline D_{t}\hat \mu_{t}(\phi)  \leq \frac{\hat \mu_{t_{k} } (\phi) - \hat \mu_{t}(\phi)}{t_{k}-t} + o(1)
\]
as $j\to\infty$. (The case where $t_k \searrow t$ is analogous.) By choosing the indices $\ell(k)$ to tend to infinity sufficiently fast, we arrange that 
\[
- \infty < \overline D_{t}\hat \mu_{t}(\phi)  \leq \frac{ \mu_{t}(\varphi_{\ell(k)}^{2}\phi) -  \mu_{t_k}(\varphi_{\ell(k)}^{2}\phi)}{t - t_k} + o(1)
\]
as $k \to \infty$. Arguing as on p. 40 in \cite{Ilmanen:ellipticreg}, we see that there are $s_{k}\in [t_k,t]$ with 
\begin{align} \label{eqn:auxbrakkefinal}
- \infty < \overline D_{t} \hat \mu_{t}(\phi)   \leq \mathcal{B}(\mu_{s_{k}},\varphi_{\ell(k)}^{2}\phi) + o(1)
\end{align}
as $k \to \infty$. In particular,
\[
\limsup_{k \to \infty} \int  \varphi_{\ell(k)}^{2} \phi  |\mathbf{H}|^{2} d\mu_{s_{k}} < \infty.
\]
The measures $\mu_{s_k} \lfloor \{x \in \R^{n+1} \setminus \{0\} : \phi (x) > 0\}$ converge to $\mu_{t}\lfloor \{x \in \R^{n+1} \setminus \{0\}: \phi(x) > 0\}$ as $k \to \infty$ by the  same  argument as on p. 41 of \cite{Ilmanen:ellipticreg}. In fact, the associated varifolds converge. It follows that 
\[
\limsup_{k \to \infty} \mathcal{B} (\mu_{s_k}, \varphi_{\ell(k)}^2 \phi) \leq \mathcal{B} (\hat \mu_t, \phi).
\]
Together with \eqref{eqn:auxbrakkefinal} this finishes the proof. 
\end{proof}

\end{proposition}

We remark that Example \ref{example:rays} shows that there is no analogue of  Proposition \ref{prop:extend-brakke-pt}  when $n = 1$. Indeed, stationary varifolds are (constant) Brakke flows. \\

The following result and its proof should be compared with the Constancy Theorem for stationary varifolds, as presented in \S 41 of \cite{Simon:GMT}.  

\begin{proposition} [Constancy theorem] \label{prop:standardlimitflow}
Let $\{ \mu_t \}_{t \geq0}$ be an integral Brakke flow in $\R^{3}$ such that $\mu(0)=\mathcal{H}^{2}\lfloor S_{1}(\xi)$. There is $T \in [0,\frac{1}{4})$ such that $\mu_t = \mathcal{H}^{2}\lfloor S_{\sqrt{1-4t}}(\xi)$ for all $t \in [0, T]$ and $\mu_t = 0$ for all  $t > T$. 
\end{proposition}
\begin{proof}
The avoidance principle for Brakke flows --- as stated in \S 10.7 of \cite{Ilmanen:ellipticreg} --- shows that 
\[
\supp \mu_t \subset S_{\sqrt{1-4t}}(\xi)
\]
for all $t \in [0,\frac{1}{4}]$. The entropy of $S_{1}(\xi)$ is less than $3/2$. The entropy decreases along the Brakke flow by Lemma 7 of \cite{Ilmanen:singularities}. Using that $\{ \mu_t \}_{t \geq 0}$ is an integral Brakke flow, we see that for almost every $t \geq 0$ the measure $\mu_t$ has an approximate tangent plane with multiplicity one at $x$ for $\mu_t$-almost every $x$. Thus, for almost every $t \geq 0$, there is a measurable subset $\Sigma_{t} \subset S_{\sqrt{1-4t}} (\xi)$ with $\mu_t = \cH^{2}\lfloor \Sigma_{t}$. 

We claim that $\Sigma_t$ --- as a varifold with multiplicity one ---  has absolutely continuous first variation. Indeed, by \S 7.2 (ii) in \cite{Ilmanen:ellipticreg}, given  $\phi \in C^\infty_c (\R^3)$ we have that $- \infty < \overline D_t \mu_t (\phi)$ for almost every $t \geq 0$. Let $\phi$ such that $\phi (x) = 1$ for all $x \in B_2 (\xi)$. Using that $\overline D_t \mu_t(\phi) \leq \mathcal{B} (\mu_t, \phi)$ for a Brakke motion, the claim follows. 

For every $X \in C^{1}_{c}(\R^{3}, \R^3)$, we have that
\[
(\delta \mu_t) (X) = \int_{S_{\sqrt{1-4t}}(\xi)} \chi_{\Sigma_{t} } \text{div}_{S_{\sqrt{1-4t}}(\xi)} X \leq c_1 \Big( \int_{S_{\sqrt{1-4t}}(\xi)} |X|^{2}  \Big)^{1/2} \leq c_2 \sup_{S_{\sqrt {1 - 4 t}} (\xi)} |X|.
\]
It follows that the perimeter of $\Sigma_t$ as a subset of $S_{\sqrt {1 - 4 t} (\xi)}$ vanishes. The Poincar\'e inequality (as in Lemma 6.4 of \cite{Simon:GMT}) shows that either $\Sigma_t$ or its complement in $S_{\sqrt {1 - 4 t}} (\xi)$ is a set of $2$--dimensional measure zero. We have thus shown that for almost every $t \geq 0$, either $\mu_t = \cH^2\lfloor S_{\sqrt{1 - 4t} } (\xi)$ or $\mu(t) = 0$.

By \S 7.2 (ii) of \cite{Ilmanen:ellipticreg}, we have that 
\[
\lim_{t \nearrow s} \mu_s (\phi) \geq \mu_s(\phi) \geq \lim_{t \searrow s}\mu_t (\phi)
\]
for all $\phi \in C_c^\infty(\R^3)$ and all $s \geq 0$. This finishes the proof.
\end{proof}


\section {Geometry in the asymptotically flat end} \label{sec:gbarg}

Consider a Riemannian metric 
\[
 g = \sum_{i, j = 1}^3 g_{ij} d x^i \otimes dx^j \qquad \text{ where } \qquad g_{ij} = \delta_{ij} + \sigma_{ij}
\]
on $\{x \in \R^3 : |x| > 1/2\}$ where
\[
|x| | \sigma_{ij}| + |x|^2 |\partial_k \sigma_{ij}|= O (|x|^{-\tau}) \qquad \text{ as } \qquad |x| \to \infty
\]
for some $\tau > 1/2$.
We denote the Euclidean background metric by 
\[
\overline g = \sum_{i, j = 1}^3 \delta_{ij} d x^i \otimes dx^j.
\]

Let $\Sigma$ be a two-sided surface in $\{x \in \R^3 : |x| > 1/2\}$. The unit normal, the second fundamental form, the trace-free second fundamental form, the mean curvature (all with respect to the outward pointing unit normal), and the induced surface measure of $\Sigma$ are denoted by $\nu, h, \mathring{h}, H$, and $\mu$ respectively. These geometric quantities can also be computed with respect to the standard Euclidean metric $\overline g$ in the chart at infinity \eqref{eqn:chartatinfinity}. To distinguish these Euclidean quantities from those with respect to the curved metric, we denote them using  an additional bar: $\overline \nu, \overline h, \mathring{\overline h}, \overline H$, and $\overline \mu$. \\

A standard computation as in e.g. \cite[p. 418]{Huisken-Ilmanen:2001} shows that we can compare geometric quantities with respect to the curved metric $g$ and the Euclidean background $\overline g$ metric according to
\begin{align}
\nu (x) - \overline \nu (x) &= O (|x|^{-\tau}), \\
{h} (x) - {\overline{h}} (x) &= O( | {h} (x)| |x|^{-\tau})+O(|x|^{-1 - \tau}),\label{eqn:euclidh}\\
H - \overline{H}(x) &= O( |h (x)| |x|^{-\tau}) +O(|x|^{-1 - \tau}), \label{eqn:euclidH} \\
\mathring {h} (x) - \mathring{\overline{h}} (x)&= O( |h (x)| |x|^{-\tau}) +O(|x|^{-1-\tau}) \label{eqn:euclidring}.
\end{align}


\section {Hawking mass of outlying  spheres} \label{sec:Hawkingoutlying}

We continue with the notation of Appendix \ref{sec:gbarg}. Let $\delta > 0$. \\

We consider a closed surface $\Sigma$ in the chart at infinity \eqref{eqn:chartatinfinity} that is geometrically close to a coordinate sphere $S_\rho (a)$ with $|a| > (1 + \delta) \rho$ and $\rho > 1$ large. More precisely, we ask that the rescaled surface 
\[
\rho^{-1} \, \Sigma = \{ \rho^{-1} \, x : x \in \Sigma \}
\]
is $C^2$ close to the boundary of a unit ball in $\{x \in \R^{3} : |x| > \delta\}$. We claim that 
\[
m_{H} (\Sigma) = o(1)
\]
as $\rho \to \infty$. To see this, we follow the strategy of G. Huisken and T. Ilmanen in their proof of the ``Asymptotic Comparison Lemma 7.4" in  \cite{Huisken-Ilmanen:2001}. We use the positivity of a term dropped in \cite{Huisken-Ilmanen:2001} in conjunction with estimates of C. De Lellis and S. M\"uller \cite{DeLellis-Muller:2005} to handle an additional technical difficulty brought about by our weaker decay assumptions $\tau > 1/2$. All integrals below are with respect to the Euclidean background metric, unless explicitly noted otherwise.

Let $r > 0$ so that 
\[
\overline \area (\Sigma) = 4 \pi r^2.
\]
Clearly, $r$ and $\rho$ are comparable. Following G. Huisken and T. Ilmanen \cite[(7.11)]{Huisken-Ilmanen:2001}, we compute 
\begin{align*}
16\pi - \int_{\Sigma} H^{2}d\mu  = 16\pi - \int_{\Sigma} \overline H^{2}  \\
 \qquad +  \int_{\Sigma} \left( - \frac 12 H^{2} \tr_{\Sigma}\sigma + 2 H g(\sigma,h) - H^{2}\sigma(\nu,\nu) + 2 H \tr_{\Sigma} ( \nabla_{\, \cdot \,}\sigma)(\nu,\, \cdot \,) - H \tr_{\Sigma} \nabla_{\nu}\sigma \right) d\mu\\
 \qquad + O\int_{\Sigma} |\sigma|^{2}|h|^{2} + O\int_{\Sigma} |\partial \sigma|^{2}.
\end{align*}
The error terms are both $O(r^{-2\tau})$, since
\[
\int_\Sigma |h|^2 d \mu = O (1).
\]
By the Gauss equation and the Gauss--Bonnet formula, 
\[
16\pi - \int_{\Sigma} \overline H^{2} = - 2 \int_{\Sigma} |\mathring{\overline h}|^{2} . 
\]
Using this in the above equation and computing as in \cite[p.\ 420] {Huisken-Ilmanen:2001}, we arrive at  
\begin{align*}
16\pi - \int_{\Sigma} H^{2}d\mu & = - 2 \int_{\Sigma} |\mathring{\overline h}|^{2}
+  \frac 1 r \int_{\Sigma} \left( H \tr_{\Sigma}\sigma - 2 H \sigma(\nu,\nu) + 4 \tr_{\Sigma} (\nabla_{\, \cdot \,}\sigma )(\nu,\, \cdot \,) - 2 \tr_{\Sigma} \nabla_{\nu}\sigma \right) d\mu\\
& \qquad + O\int_{\Sigma}  |H - 2/r| ( H |\sigma| + |\partial \sigma|) + O \int_{\Sigma} H |\mathring h| |\sigma| + O(r^{-2\tau}).
\end{align*}
Finally, integrating by parts as in Huisken--Ilmanen (7.15), we find  
\[
 \int_{\Sigma} 2 \tr_{\Sigma} (\nabla_{\, \cdot \,}\sigma)(\nu,\, \cdot \,) d\mu = \int_{\Sigma} \left(2H \sigma(\nu,\nu) - H \tr_{\Sigma}\sigma \right)d\mu + O \int_\Sigma |\mathring h| |\sigma| 
\]
so that
\begin{align*}
16\pi - \int_{\Sigma} H^{2}d\mu  = - 2 \int_{\Sigma} |\mathring{\overline h}|^{2}
+ O\int_{\Sigma}  | H - 2/r | (H|\sigma| + |\partial \sigma|) + O \int_{\Sigma} H |\mathring h| |\sigma| 
 + O(r^{-2\tau})\\
 +  \frac 2 r \int_{\Sigma} \left(  \tr  (\nabla_{\, \cdot \,}\sigma)(\nu,\, \cdot \,) -  \tr \nabla_{\nu}\sigma \right) d\mu.
\end{align*}
Using that the scalar curvature is integrable and that $\Sigma$ is outlying and divergent as $r \to \infty$, we see that the ``mass integral" on the second line is $o(r^{-1})$.
Using \eqref{eqn:euclidH} and the (trivial) estimate $|h(x)| = O (r^{-1})$, we may rewrite the above expression as 
\begin{align*}
16\pi - \int_{\Sigma} H^{2}d\mu  \\ = - 2 \int_{\Sigma} |\mathring{\overline h}|^{2}
 + O\int_{\Sigma}  | \overline H - 2/r| (|\sigma|/r + |\partial \sigma|) + O \int_{\Sigma} H |\mathring h| |\sigma| 
 + O \int_{\Sigma} ( |\sigma|/r + |\partial \sigma|)^{2}
 + o(r^{-1})
\end{align*}
It is clear that this additional error term is $o(r^{-1})$. Simplifying, we find  
\begin{align*}
16\pi - \int_{\Sigma} H^{2}d\mu  = - 2 \int_{\Sigma} |\mathring{\overline h}|^{2}
 + O\left( r^{-1-\tau}\int_{\Sigma}  | \overline H - 2/r | \right)+ O \left( r^{-1-\tau} \int_{\Sigma}  |\mathring h| \right) 
 + o(r^{-1})
\end{align*}
Using now that $|\mathring h (x)| = |\mathring{\overline h} (x)| + O(r^{-1-\tau})$ by \eqref{eqn:euclidring}, we obtain  
\begin{align*}
16\pi - \int_{\Sigma} H^{2}d\mu  = - 2 \int_{\Sigma} |\mathring{\overline h}|^{2}
  + O\left( r^{-1-\tau}\int_{\Sigma}  | \overline H - 2/r| \right)+ O \left( r^{-1-\tau} \int_{\Sigma}  |\mathring{\overline h}| \right) 
  + o(r^{-1}).
\end{align*}
Using H\"older's inequality, we find
\begin{align*}
16\pi - \int_{\Sigma} H^{2}d\mu & \leq -  \int_{\Sigma} |\mathring{\overline h}|^{2} + O\left( r^{-1-\tau}\int_{\Sigma}  | \overline H - 2/r| \right) + o(r^{-1}).
\end{align*}
Using now the estimate
\begin{align}\label{eq:DM-H-vs-tfreeh}
\int_\Sigma ( \overline H - 2/r )^2 \leq c  \int_{\Sigma} |\mathring{\overline h}|^{2}
\end{align}
due to C. De Lellis and S. M\"uller \cite{DeLellis-Muller:2005} where $c > 0$ is a universal constant, it follows that 
\[
O \left( r^{-1-\tau}\int_{\Sigma}  | \overline H - 2/r | \right) \leq  \int_{\Sigma} |\mathring{\overline h}|^{2} + O(r^{-2\tau}). 
\]
Thus
\[
16\pi - \int_{\Sigma} H^{2}d\mu \leq o(r^{-1}) \qquad \text{ or, equivalently, } \qquad m_H (\Sigma) = o(1)
\]
as $r \to \infty$.

\begin{remark}
Since $\Sigma$ is geometrically close to $S_{\rho}(a)$ it is in particular convex. There are two alternative proofs of \eqref{eq:DM-H-vs-tfreeh} in this case. One  is due to G.\ Huisken and uses inverse mean curvature flow of mean-convex, star-shaped regions in $\R^{3}$ --- see Theorem 3.3 in \cite{Perez:thesis}. A second proof is due to D.\ Perez \cite[Theorem 3.1]{Perez:thesis}, who proves \eqref{eq:DM-H-vs-tfreeh} for convex hypersurfaces in $\R^{n+1}$ and proceeds via integration by parts with an appropriately chosen solution to the Poisson equation. 
\end{remark}


\bibliography{bib} 

\providecommand{\bysame}{\leavevmode\hbox to3em{\hrulefill}\thinspace}
\providecommand{\MR}{\relax\ifhmode\unskip\space\fi MR }
\providecommand{\MRhref}[2]{%
  \href{http://www.ams.org/mathscinet-getitem?mr=#1}{#2}
}
\providecommand{\href}[2]{#2}
\begin{thebibliography}{10}

\bibitem{ADM:1961}
Richard Arnowitt, Stanley Deser, and Charles Misner, \emph{Coordinate
  invariance and energy expressions in general relativity.}, Phys. Rev. (2)
  \textbf{122} (1961), 997--1006. \MR{0127946}

\bibitem{Bartnik:1986}
Robert Bartnik, \emph{The mass of an asymptotically flat manifold}, Comm. Pure
  Appl. Math. \textbf{39} (1986), no.~5, 661--693. \MR{849427}

\bibitem{Bavard-Pansu:1986}
Christophe Bavard and Pierre Pansu, \emph{Sur le volume minimal de {${\bf
  R}^2$}}, Ann. Sci. \'Ecole Norm. Sup. (4) \textbf{19} (1986), no.~4,
  479--490. \MR{875084 (88b:53048)}

\bibitem{Brakke}
Kenneth~A. Brakke, \emph{The motion of a surface by its mean curvature},
  Mathematical Notes, vol.~20, Princeton University Press, Princeton, N.J.,
  1978. \MR{485012}

\bibitem{Bray:1997}
Hubert Bray, \emph{The {P}enrose inequality in general relativity and volume
  comparison theorems involving scalar curvature}, ProQuest LLC, Ann Arbor, MI,
  1997, Thesis (Ph.D.)--Stanford University. \MR{2696584}

\bibitem{offcenter}
Simon Brendle and Michael Eichmair, \emph{Large outlying stable constant mean
  curvature spheres in initial data sets}, Invent. Math. \textbf{197} (2014),
  no.~3, 663--682. \MR{3251832}

\bibitem{Carlotto:2014}
Alessandro Carlotto, \emph{Rigidity of stable minimal hypersurfaces in
  asymptotically flat spaces}, preprint,
  \url{http://arxiv.org/pdf/1403.6459.pdf} (2014).

\bibitem{mineffectivePMT}
Alessandro Carlotto, Otis Chodosh, and Michael Eichmair, \emph{Effective
  versions of the positive mass theorem}, to appear in Invent. Math.,
  \url{http://arxiv.org/abs/1503.05910} (2016).

\bibitem{Carlotto-Schoen}
Alessandro Carlotto and Richard Schoen, \emph{Localized solutions of the
  {E}instein constraint equations}, to appear in Invent. Math.,
  \url{http://arxiv.org/abs/1407.4766} (2014).

\bibitem{Cederbaum-Nerz:2015}
Carla Cederbaum and Christopher Nerz, \emph{Explicit riemannian manifolds with
  unexpectedly behaving center of mass}, Ann. Henri Poincar{\'e} \textbf{16}
  (2015), no.~7, 1609--1631. \MR{3356098}

\bibitem{Chodosh:large-iso}
Otis Chodosh, \emph{Large isoperimetric regions in asymptotically hyperbolic
  manifolds}, Comm. Math. Phys. \textbf{343} (2016), no.~2, 393--443.
  \MR{3477343}

\bibitem{Christodoulou-Yau:1988}
Demetrios Christodoulou and Shing-Tung Yau, \emph{Some remarks on the
  quasi-local mass}, Mathematics and general relativity ({S}anta {C}ruz, {CA},
  1986), Contemp. Math., vol.~71, Amer. Math. Soc., Providence, RI, 1988,
  pp.~9--14. \MR{954405 (89k:83050)}

\bibitem{CorvinoGerekGreenbergKrummel}
Justin Corvino, Aydin Gerek, Michael Greenberg, and Brian Krummel, \emph{On
  isoperimetric surfaces in general relativity}, Pacific J. Math. \textbf{231}
  (2007), no.~1, 63--84. \MR{2304622 (2008k:53168)}

\bibitem{DeLellis-Muller:2005}
Camillo De~Lellis and Stefan M{\"u}ller, \emph{Optimal rigidity estimates for
  nearly umbilical surfaces}, J. Differential Geom. \textbf{69} (2005), no.~1,
  75--110. \MR{MR2169583 (2006e:53078)}

\bibitem{stablePMT}
Michael Eichmair and Jan Metzger, \emph{On large volume preserving stable {CMC}
  surfaces in initial data sets}, J. Differential Geom. \textbf{91} (2012),
  no.~1, 81--102. \MR{2944962}

\bibitem{isostructure}
\bysame, \emph{Large isoperimetric surfaces in initial data sets}, J.
  Differential Geom. \textbf{94} (2013), no.~1, 159--186. \MR{3031863}

\bibitem{hdiso}
\bysame, \emph{Unique isoperimetric foliations of asymptotically flat manifolds
  in all dimensions}, Invent. Math. \textbf{194} (2013), no.~3, 591--630.
  \MR{3127063}

\bibitem{Fan-Shi-Tam:2009}
Xu-Qian Fan, Yuguang Shi, and Luen-Fai Tam, \emph{Large-sphere and small-sphere
  limits of the {B}rown-{Y}ork mass}, Comm. Anal. Geom. \textbf{17} (2009),
  no.~1, 37--72. \MR{2495833 (2010e:53132)}

\bibitem{Flores-Nardulli:2014}
Abraham~Mu{\~n}oz Flores and Stefano Nardulli, \emph{Continuity and
  differentiability properties of the isoperimetric profile in complete
  noncompact {R}iemannian manifolds with bounded geometry},
  http://arxiv.org/abs/1404.3245 (2014).

\bibitem{Gilbarg-Trudinger:1998}
David Gilbarg and Neil~S. Trudinger, \emph{Elliptic partial differential
  equations of second order}, Classics in Mathematics, Springer-Verlag, Berlin,
  2001, Reprint of the 1998 edition. \MR{1814364 (2001k:35004)}

\bibitem{Huang:2010}
Lan-Hsuan Huang, \emph{Foliations by stable spheres with constant mean
  curvature for isolated systems with general asymptotics}, Comm. Math. Phys.
  \textbf{300} (2010), no.~2, 331--373. \MR{2728728 (2012a:53045)}

\bibitem{Huang:2011}
\bysame, \emph{On the center of mass in general relativity}, Fifth
  {I}nternational {C}ongress of {C}hinese {M}athematicians. {P}art 1, 2, AMS/IP
  Stud. Adv. Math., 51, pt. 1, vol.~2, Amer. Math. Soc., Providence, RI, 2012,
  pp.~575--591. \MR{2908093}

\bibitem{Huisken:iso-mass}
Gerhard Huisken, \emph{An isoperimetric concept for mass and quasilocal mass},
  Oberwolfach Rep., no.~2, 2006, pp.~87--88.

\bibitem{Huisken:MM-iso-mass-video}
\bysame, \emph{An isoperimetric concept for the mass in {G}eneral
  {R}elativity}, video available at \url{http://video.ias.edu/node/234}, March
  2009.

\bibitem{Huisken-Ilmanen:2001}
Gerhard Huisken and Tom Ilmanen, \emph{The inverse mean curvature flow and the
  {R}iemannian {P}enrose inequality}, J. Differential Geom. \textbf{59} (2001),
  no.~3, 353--437. \MR{1916951}

\bibitem{Huisken-Yau:1996}
Gerhard Huisken and Shing-Tung Yau, \emph{Definition of center of mass for
  isolated physical systems and unique foliations by stable spheres with
  constant mean curvature}, Invent. Math. \textbf{124} (1996), no.~1-3,
  281--311. \MR{1369419}

\bibitem{Ilmanen:ellipticreg}
Tom Ilmanen, \emph{Elliptic regularization and partial regularity for motion by
  mean curvature}, Mem. Amer. Math. Soc. \textbf{108} (1994), no.~520, x+90.
  \MR{1196160 (95d:49060)}

\bibitem{Ilmanen:singularities}
\bysame, \emph{Singularities of mean curvature flow of surfaces}, preprint
  (1995).

\bibitem{Jauregui-Lee:2016}
Jeffrey~L. Jauregui and Dan~A. Lee, \emph{Lower semicontinuity of mass under
  {$C^0$} convergence and {H}uisken's isoperimetric mass}, preprint,
  \url{http://arxiv.org/abs/1602.00732} (2016).

\bibitem{Ji-Shi-Zhu:2015}
Dandan Ji, Yuguang Shi, and Bo~Zhu, \emph{Exhaustion of isoperimetric regions
  in asymptotically hyperbolic manifolds with scalar curvature $r \geq -6$},
  preprint, \url{http://arxiv.org/abs/1512.02732}.

\bibitem{Ma:2016}
Shiguang Ma, \emph{On the radius pinching estimate and uniqueness of the cmc
  foliation in asymptotically flat $3$-manifolds}, Adv. Math. \textbf{288}
  (2016), 942--984. \MR{MR3436403}

\bibitem{Marques-Neves:Willmore}
Fernando Marques and Andr{\'e} Neves, \emph{Min-max theory and the willmore
  conjecture}, Ann. of Math. (2) \textbf{179} (2014), no.~2, 683--782.

\bibitem{Nardulli:2009}
Stefano Nardulli, \emph{The isoperimetric profile of a smooth {R}iemannian
  manifold for small volumes}, Ann. Global Anal. Geom. \textbf{36} (2009),
  no.~2, 111--131. \MR{2529468}

\bibitem{Nerz:2014}
Christopher Nerz, \emph{Foliations by stable spheres with constant mean
  curvature for isolated systems without asymptotic symmetry}, Calc. Var.
  Partial Differential Equations \textbf{54} (2015), no.~2, 1911--1946.
  \MR{3396437}

\bibitem{Perez:thesis}
Daniel Perez, \emph{On nearly umbilical hypersurfaces}, Ph.D. thesis,
  University of Zurich, available at
  \url{http://user.math.uzh.ch/delellis/uploads/media/Daniel.pdf}, 2011.

\bibitem{Qing-Tian:2007}
Jie Qing and Gang Tian, \emph{On the uniqueness of the foliation of spheres of
  constant mean curvature in asymptotically flat 3-manifolds}, J. Amer. Math.
  Soc. \textbf{20} (2007), no.~4, 1091--1110. \MR{2328717}

\bibitem{Ros:2005}
Antonio Ros, \emph{The isoperimetric problem}, Global theory of minimal
  surfaces, Clay Math. Proc., vol.~2, Amer. Math. Soc., Providence, RI, 2005,
  pp.~175--209. \MR{2167260 (2006e:53023)}

\bibitem{PMT1}
Richard Schoen and Shing-Tung Yau, \emph{On the proof of the positive mass
  conjecture in general relativity}, Comm. Math. Phys. \textbf{65} (1979),
  no.~1, 45--76. \MR{526976}

\bibitem{Shi:2016}
Yuguang Shi, \emph{The isoperimetric inequality on asymptotically flat
  manifolds with nonnegative scalar curvature}, to appear in IMRN, doi:
  10.1093/imrn/rnv395 (2016).

\bibitem{Simon:GMT}
Leon Simon, \emph{Lectures on geometric measure theory}, Proceedings of the
  Centre for Mathematical Analysis, Australian National University, vol.~3,
  Australian National University, Centre for Mathematical Analysis, Canberra,
  1983.

\bibitem{White:size-singular}
Brian White, \emph{The size of the singular set in mean curvature flow of
  mean-convex sets}, J. Amer. Math. Soc. \textbf{13} (2000), no.~3, 665--695
  (electronic). \MR{1758759 (2001j:53098)}

\bibitem{White:regularityMCF}
\bysame, \emph{A local regularity theorem for mean curvature flow}, Ann. of
  Math. (2) \textbf{161} (2005), no.~3, 1487--1519. \MR{2180405 (2006i:53100)}

\bibitem{Witten:1981}
Edward Witten, \emph{A new proof of the positive energy theorem}, Comm. Math.
  Phys. \textbf{80} (1981), no.~3, 381--402. \MR{626707 (83e:83035)}

\end{thebibliography}
\bibliographystyle{amsplain}
\end{document}